\documentclass[a4paper, 12pt]{article}
\usepackage{times,amsmath,amsbsy,amssymb,amsthm}
\usepackage{graphicx}
\usepackage{dsfont}
\numberwithin{equation}{section}
\newtheorem{theorem}{Theorem}[section]
\newtheorem{lemma}{Lemma}[section]
\newtheorem{remark}{Remark}[section]
\newtheorem{corollary}{Corollary}[section]

\title{Quasi-optimal convergence rate for   adaptive mixed finite
element methods
\thanks{This work was supported in part by the  Natural Science Foundation of China
(10771150), the National Basic Research Program of China
(2005CB321701), and the Program for New Century Excellent Talents in
University (NCET-07-0584)}}

\author{Shaohong Du$^{1,2}$\thanks{Email: shaohongdu@gmail.com }\hspace{15mm}Xiaoping
Xie$^{1}$\thanks{Corresponding author. Email:
xpxiec@gmail.com}\\
 { \small $_1$ \ School of Mathematics, Sichuan University, Chengdu 610064, China}\ \ \ \ \ \ \\
 { \small $_2$ \ School of Science, Chongqing Jiaotong University, Chongqing 400047, China}
 }

\begin{document}
 \date{}
 \maketitle
\maketitle
\begin{small}
  {\bf{Abstract.} \rm{For  adaptive mixed finite element methods
  (AMFEM), we first introduce the
 data oscillation to analyze,  without
   the restriction that  the inverse of the coefficient matrix of the
  partial differential equations (PDEs) is a piecewise polynomial matrix, efficiency of the a posteriori error
  estimator  Presented by Carstensen [Math. Comput., 1997, 66: 465-476] for Raviart-Thomas,
  Brezzi-Douglas-Morini, Brezzi-Douglas-Fortin-Marini elements.
  Second, we prove that the sum of the
  stress variable error in a weighted norm and the scaled error
  estimator is of geometric decay, namely, it reduces with a fixed
   factor between two successive
  adaptive loops, up to an oscillation of the right-hand side term
  of the PDEs. Finally, with the help of this geometric decay, we
  show that the stress variable error in a weighted norm plus
  the oscillation of data yields a decay rate in
  terms of the number of degrees of freedom as dictated by the best
  approximation for this combined nonlinear quantity.}}
\end{small}

\begin{it} Key words.\end{it} mixed finite element, error reduction, convergence, optimal
 cardinality, adaptive algorithm

\begin{it} AMS subject classifications.\end{it} 65N30, 65N50, 65N15,
65N12, 41A25

\setcounter{remark}{0} \setcounter{lemma}{0} \setcounter{theorem}{0}
\setcounter{section}{0} \setcounter{equation}{0}
\section {Introduction and main results}

Adaptive methods for the numerical solution of the PDEs are now
standard tools in science and engineering to achieve better accuracy
with minimum degrees of freedom. The adaptive procedure consists of
loops of the form
\begin{equation}\label{quasi-optimality 0}
SOLVE\rightarrow ESTIMATE\rightarrow MARK\rightarrow REFINE.
\end{equation}
A posteriori error estimation (ESTIMATE) is
  an essential ingredient of adaptivity. We refer
  to \cite{Babuska;Rheinboldt1,Babuska;Rheinboldt2,Anis;Oden,
Babuska;Strouboulis,Bangerth;Rannacher,
Carstensen;Hu,Eriksson;Estep;Hansbo,Verfurth,Du;Xie} for related
work on this topic. The analysis of convergence and optimality of
the above whole algorithm is still in its infancy. In recent years,
there have been some results for the  standard adaptive finite
element method
\cite{Dorfler,Morin2000,Morin2003,Morin;Nochetto;Siebert,
Cascon;Kreuzer;Nochetto;Siebert}. In
\cite{Carstensen;Hoppe,Chen;Holst;Xu,Becker,Carstensen;Rabus},
convergence analysis has been carried out for the AMFEM.

Let $\Omega$ be a bounded polygonal in ${\mathbb{R}}^{2}$. We
consider the following homogeneous Dirichlet boundary value problem
for a second order elliptic PDE:
\begin{equation}\label{quasi-optimality 1}
 \left \{ \begin{array}{ll}
  -\mbox{div}(A\nabla u)=f\ \ \  & \mbox{in}\ \ \Omega,\\
 \ \hspace{17mm}u=0 & \mbox{on}\ \ \ \partial{\Omega},
 \end{array}\right.
\end{equation}
where $A\in L^{\infty}(\Omega;{\mathbb{R}}^{2\times 2})$ is a
symmetric and uniformly positive definite matrix, and $f\in
L^{2}(\Omega)$. The choice of boundary conditions is made for ease
of presentation, since similar results are valid for other boundary
conditions. In \cite{Carstensen0}, Carstensen presented an a posteriori
error estimate for the mixed finite element method of
(\ref{quasi-optimality 1}), and analyzed its efficiency under some
restriction of the coefficient matrix. In this paper,  we shall prove its
efficiency without
 the restriction, but at the expense of introducing 
data oscillation, which is right a component of the new concept of
error ({\it the total error}) (see
\cite{Cascon;Kreuzer;Nochetto;Siebert}).

To summarize the first main result, let $\mathcal{T}_{h},M_{h}\times
L_{h},(p_{h},u_{h}),\eta_{h,\kappa},\widetilde{{\rm osc}}_{h}$
denote the meshes, a pair of finite element spaces, a pair of
corresponding discrete solutions, the estimators and oscillations in
turn. We avoid the  assumption that $A^{-1}p_{h}$
is a polynomial on each element, which is required in \cite{Carstensen0} for the proof of
efficiency of the a posteriori error estimator, and obtain the following efficient estimates for
the Raviart-Thomas, the Brezzi-Douglas-Marini, or the
Brezzi-Douglas-Fortin-Marini elements
\begin{equation*}
\eta_{h,\kappa}^{2}\lesssim||A^{-1/2}(p-p_{h})||_{L^{2}(\Omega)}^{2}+
||h^{\kappa}{\rm
div}(p-p_{h})||_{L^{2}(\Omega)}^{2}+||u-u_{h}||_{L^{2}(\Omega)}^{2}+\widetilde{{\rm
osc}}_{h}^{2}.
\end{equation*}

Secondly, we shall analyze convergence and optimality of the AMFEM
of the form (\ref{quasi-optimality 0}). Here we only concern the stress
variable error, which is of interest in many applications.

The convergence analysis of the adaptive finite element method
(AFEM) is very recent, it started with D\"{o}fler \cite{Dorfler},
who introduced a crucial marking, and proved   the strict energy
error reduction of the standard AFEM for the Laplacian under the
condition that the initial mesh $\mathcal{T}_{0}$ satisfies a
fineness assumption. Morin, Nochetto and siebert
\cite{Morin2000,Morin;Nochetto;Siebert} showed that such strict
energy error reduction can not be expected in general. Introducing
the concept of data oscillation and the interior node property, they
proved convergence of the standard AFEM without fineness restriction
on $\mathcal{T}_{0}$ which is valid only for $A$ in (2) being
piecewise constant on $\mathcal{T}_{0}$. Inspired by the work by
Chen and Feng \cite{Chen;Feng}, Mekchay and Nochetto
\cite{Morin;Nochetto;Siebert} extended this result to general second
elliptic operators and proved that the standard AFEM is a
contraction for the total error, namely the sum of the energy error
and oscillation. Recently, Cascon, Kreuzer, Nochetto and Seibert
\cite{Cascon;Kreuzer;Nochetto;Siebert} presented a new error notion,
the so-called {\it quasi-error}, namely the sum of the energy error
and the scaled estimator, and showed without the interior node
property for the self-adjoint second elliptic problem that the
quasi-error is strictly reduced by the standard AFEM even though
each term may not be.

However,  for convergence of the AMFEM, present woks are done only
for the Laplacian for the lowest order Raviart-Thomas elements
\cite{Carstensen;Hoppe,Becker} and any order Raviart-Thomas and
Brezzi-Douglas-Marini elements \cite{Chen;Holst;Xu}. Since the
approximation of mixed finite element methods is a saddle point of
the corresponding energy, there is no orthogonality available, as is
one of main difficulties for convergence of the AMFEM. In this
paper, motivated by the two new notions of the error, by
establishing a quasi-orthogonality result (see Theorem \ref{thm 4.5}) we
proved that the AMFEM is a contraction with respect to the sum of
the stress variable error in a weighted norm and the scaled error
estimator, which is also called the {\it quasi-error}.

To summarize the second main result, let
$\{\mathcal{T}_{k},(M_{k},L_{k}),(p_{k},u_{k}), \eta_{k},
osc_{k}\}_{k\geq0}$ with $\text{div}M_{k}=L_{k}$ be the sequence of
the meshes, a pair of finite element spaces, a pair of corresponding
discrete solutions, the estimators and oscillations produced by the
AMFEM in the $k$-th step.  We prove in Section 5 that the
quasi-error uniformly reduces with a fixed rate between two
successive meshes, up to an oscillation of data $f$, namely
\begin{equation*}
\mathcal{E}_{k+1}^{2}+\gamma\eta_{k+1}^{2}\leq\alpha^{2}(\mathcal{E}_{k}^{2}+\gamma\eta_{k}^{2})+
C{\rm osc}^{2}(f,\mathcal{T}_{k}),
\end{equation*}
where $\alpha\in(0,1)$, $\gamma>0$, 
$$\mathcal{E}_{k}^{2}
:=||A^{-1/2}(p-p_{k})||_{L^{2}(\Omega)}^{2}+||h_{k}{\rm
div}(p-p_{k})||_{L^{2}(\Omega)}^{2},$$
and ${\rm
osc}(f,\mathcal{T}_{k})$ is the oscillation of $f$ over
$\mathcal{T}_{k}$ (see Section 2.5).  We point out here that in some cases, even though
the stress variable error is monotone, strict error reduction may fail. For instance,
when $p_{k}=p_{k+1}$ and $f\in L_{k}$, from the second equation of (\ref{quasi-optimality 2.3}) it follows $\text{div}p_{k}=-f$. Then    it holds $\mathcal{E}_{k}^{2}
=\mathcal{E}_{k+1}^{2}$. On the
other hand, the residual estimator $\eta_{k} :=\eta_{k}(p_{k},
\mathcal{T}_{k})$ displays strict reduction when $p_{k}=p_{k+1}$ but
no monotone behavior in general.

Besides convergence, optimality is another important issue in AFEM
which was first addressed by Binev, Dahmen, DeVore
\cite{Binev;Dahmen;DeVore} and further studied by Stevenson
\cite{Stevenson}, who showed optimality without additional
coarsening  required in \cite{Binev;Dahmen;DeVore}. Both papers
\cite{Binev;Dahmen;DeVore,Stevenson} are restricted to Laplace
operator and rely on suitable marking by data oscillation and the
interior node property. Cascon, Kreuzer, Nochetto and Seibert
\cite{Cascon;Kreuzer;Nochetto;Siebert} succeeded in establishing
quasi-optimality of the AFEM without both the assumption of the
interior node property and marking by data oscillation for the
self-adjoint second elliptic operator.

Since all decisions of the AMFEM in MARK are based on the estimator
$\eta_{k}$, a decay rate for the true error is closely related to
the quality of the estimator, which is described by the global lower
bound
\begin{equation*}
\eta_{k}^{2}\lesssim\mathcal{E}_{k}^{2}+{\rm osc}_{k}^{2}.
\end{equation*}
Hereafter, following the idea in
\cite{Cascon;Kreuzer;Nochetto;Siebert}, we refer to the square root
of right-hand side above as the {\it total error}
\cite{Mekchay;Nochetto}. The lower bound demonstrates that the
estimator is controlled by the error with weights except up to an
oscillation term and one can observe the difference between
$\mathcal{E}_{k}$ and $\eta_{k}$ only when oscillation is large.
Furthermore, from the upper bound
$\mathcal{E}_{k}^{2}\lesssim\eta_{k}^{2}$ and ${\rm
osc}_{k}^{2}\leq\eta_{k}^{2}$ it follows $\mathcal{E}_{k}^{2}+{\rm
osc}_{k}^{2}\lesssim \eta_{k}^{2}$. This implies that the total
error , which is the quantity reduced
by the AMFEM, is controlled by the estimator. Since the estimator itself is an upper bound for the
quasi-error, in view of the global lower bound it holds
\begin{equation}
\mathcal{E}_{k}^{2}+{\rm osc}_{k}^{2}\approx\eta_{k}^{2}\approx
\mathcal{E}_{k}^{2}+\gamma\eta_{k}^{2}.
\end{equation}

In short, the behavior of the AMFEM is intrinsically bonded to the
total error, which measures the approximability of both the flux
$p=A\nabla u$ and data encoded in the oscillation term. Note that
when $A^{-1}p_h$ is a piecewise polynomial vector, oscillation will reduce
to approximation of the right-hand side term $f$ of
(\ref{quasi-optimality 1}) (see Section 2.5).  In general cases, approximation of data
$A$ appeared in ${\rm osc}_{k}^{2}$ couples in
nonlinear fashion with the discrete solutions $p_{k}$.

In Section 6, we shall introduce two approximation classes $\mathbb{A}_{s}$ and
$\mathcal{A}_{0}^{s}$  based on the total error and
oscillation of $f$, respectively. Using a quasi-monotonicity
property of oscillation and a localized discrete upper bound, we
prove the following {\rm quasi-optimal} convergence rate for the
AMFEM in terms of DOFs by assuming the marking parameter
$\theta\in(0,\theta_{*})$ with $0<\theta_{*}<1$ (see Theorem \ref{thm 6.10}):
\begin{equation*}
(\mathcal{E}_{N}^{2}+{\rm osc}_{N}^{2})^{1/2}\leq
C^{s}\Theta^{s}(s,\theta)(|(p,f,A)|_{s}+||f||_{\mathcal{A}_{0}^{s}})(\#\mathcal{T}_{N}-
\#\mathcal{T}_{0})^{-s}.
\end{equation*}

The rest of this paper is organized as follows. In Section 2, we
shall give some preliminaries and details on notations. Some
auxiliary results are included in Section 3 for later usage. Section
4 is devoted to the analysis of efficiency of the a posteriori error
estimator. The proof of convergence for the AMFEM is placed in
Section 5. Finally, we shall prove the quasi-optimal
convergence rate for the AMFEM in Section 6 and give
conclusions in Section 7.

\setcounter{remark}{0}

\setcounter{lemma}{0} \setcounter{theorem}{0}
\setcounter{section}{1} \setcounter{equation}{0}
\section{Preliminaries and notations}
\subsection{Weak formulation}
By splitting (\ref{quasi-optimality 1}) into two equations, the
mixed formulation is given as
\begin{equation}\label{quasi-optimality 2.1}
 \left \{ \begin{array}{ll}
  -\mbox{div}\ p=f \ \ {\rm and}\ \ p=A\nabla u  & \mbox{in}\;  \ \ \Omega\\
\hspace{9mm} u=0 & \mbox{on}\; \ \ \partial{\Omega}.
 \end{array}\right.
\end{equation}
Since the coefficient matrix $A$ is symmetric and uniformly positive
definite, by the Lax-Milgram theorem, there exists a unique solution
$u\in H_{0}^{1}(\Omega)$ to the problem (\ref{quasi-optimality
2.1}). Moreover, the weak formulation of (\ref{quasi-optimality
2.1}) reads as: Find $(p,u)\in H({\rm div},\Omega)\times
L^{2}(\Omega)$ such that
\begin{equation}\label{quasi-optimality 2.2}
\begin{array}{lll}
\displaystyle(A^{-1}p,q)_{0,\Omega}+({\rm div}\ q,u)_{0,\Omega}=0\ \
&{\rm for\ all}&\displaystyle\ q\in H({\rm
div},\Omega) ,\vspace{2mm}\\
\displaystyle({\rm div}\ p, v)_{0,\Omega}=-(f,v)_{0,\Omega}\ \ &{\rm
for\ all}&\displaystyle\ v\in L^{2}(\Omega) ,
\end{array}
\end{equation}
where $H({\rm div},\Omega) :=\{q\in L^{2}(\Omega)^{2}:
{\rm div}\ q\in L^{2}(\Omega)\}$ is endowed with the norm given by
$||q||_{H({\rm div},\Omega)}^{2} :=||q||_{L^{2}(\Omega)}^{2}+||{\rm
div}\ q||_{L^{2}(\Omega)}^{2}$, and $(\cdot,\cdot)_{0,\Omega}$ denotes
$L^{2}$ inner product on $\Omega$.

For a given  shape-regular triangulation $\mathcal {T}_{h}$ of $\Omega$
into triangles, let $M_{h}$ and $L_{h}$ denote  finite
dimensional subspaces of $H({\rm div},\Omega)$ and $L^{2}(\Omega)$,
respectively.
In the step $SOLVE$ a mixed finite element method reads as: Find
$(p_{h},u_{h})\in M_{h}\times L_{h}$ such that
\begin{equation}\label{quasi-optimality 2.3}
\begin{array}{lll}
\displaystyle(A^{-1}p_{h},q_{h})_{0,\Omega}+({\rm div}\
q_{h},u_{h})_{0,\Omega}=0\ \ &{\rm for\ all}&\displaystyle\
q_{h}\in M_{h},\vspace{2mm}\\
\displaystyle({\rm div}\ p_{h},
v_{h})_{0,\Omega}=-(f_{h},v_{h})_{0,\Omega}\ \ &{\rm for\
all}&\displaystyle\ v_{h}\in L_{h},
\end{array}
\end{equation}
where $f_{h}$ is the $L^{2}-$ projection of $f$ over $L_{h}$.

It is well-known that existence and uniqueness of the
solution of (\ref{quasi-optimality 2.2}) hold true, and that the
discrete problem (\ref{quasi-optimality 2.3}) has a unique solution
when a discrete inf-sup-condition is satisfied by the discrete
spaces $M_{h}$ and $L_{h}$ (cf. \cite{Brezzi;Fortin}). So we are
interested in controlling stress variable error $\epsilon
:=p-p_{h}\in H({\rm div},\Omega)$ and displacement error $e
:=u-u_{h}\in L^{2}(\Omega)$, and suppose that the module $SOLVE$
outputs a pair of discrete solutions over $\mathcal{T}_{h}$, namely,
$(p_{h},u_{h})=SOLVE(\mathcal{T}_{h})$.

\subsection{Mixed finite elements}
We consider some 
well-known mixed finite elements
for the discretization problem (\ref{quasi-optimality 2.3}), such as
Raviart-Thomas (RT) elements, Brezzi-Douglas-Marini (BDM) elements
and Brezzi-Douglas-Fortin-Marini (BDFM) elements
\cite{Raviart;Thomas,Brezzi;Douglas,Brezzi;Fortin}, which are
briefly described for all triangle $T\in\mathcal{T}_{h}$ by some
$D_{l}(T)\subset C(T)$ and $M_{l}(T)\subset C(T)^{2}$ given in the
following table.
\begin{center}{\hbox{\hspace{37mm} Examples for mixed finite elements}}
\scriptsize
\begin{tabular}{|c|c|c|} \hline
Element&$M_{l}(T)$&$D_{l}(T)$\\ \hline
RT&$P_{l}^{2}$+$P_{l}x$&$P_{l}$\\ \hline BDM&$P_{l+1}^{2}$&$P_{l}$\\
\hline BDFM&$\{q\in P_{l+1}^{2}: (q\cdot\nu)|_{\partial T}\in
R_{l}(\partial T)\}$&$P_{l}$ \\ \hline
\end{tabular}
\end{center}
Here $P_{l}$ denotes the set of
polynomials of total degree $\leq l$ and $R_{l}(\partial T)$ denotes the set of
polynomials of degree at most $l$ on each edge of $T$ (not
necessary continuous).

By using the above sets $M_{l}(T)$ and
$D_{l}(T)$, the discrete spaces $M_{h}$ and $L_{h}$ are given by
\begin{equation*}
\begin{array}{lll}
\displaystyle M_{h}&:&=\displaystyle\{q_{h}\in H({\rm div},\Omega) :
q_{h}|_{T}\in
M_{l}(T)\ \text{ for }  T\in\mathcal{T}_{h}\},\vspace{2mm}\\
\displaystyle L_{h}&:&=\displaystyle\{v_{h}\in L^{2}(\Omega):
v_{h}|_{T}\in D_{l}(T)\ \text{ for }  T\in\mathcal{T}_{h}\}.
\end{array}
\end{equation*}

\subsection{Assumption on $\mathcal{T}_{h}$}
Let $\mathcal{T}_{h}$ be a shape regular triangulation  in the sense of
\cite{Ciarlet} which satisfies the angle condition, namely there
exists a constant $c_{1}$ such that for all $T\in\mathcal{T}_{h}$
\begin{equation*}
c_{1}^{-1}h_{T}^{2}\leq|T|\leq c_{1}h_{T}^{2},
\end{equation*}
where $h_{T} :={\rm diam}(T)$, and$|T|$ is the area of $T$.

Let $\varepsilon_{h}$ denote the set of element edges in
$\mathcal{T}_{h}$, $J(v)|_{E} :=(v|_{T_{+}})|_{E}-(v|_{T_{-}})|_{E}$
denote the jump of $v\in H^{1}(\bigcup\mathcal{T}_{h})$ over an
interior edge $E :=T_{+}\cap T_{-}$ of length $h_{E}: ={\rm
diam}(E)$, shared by the two neighboring (closed) triangles
$T_{\pm}\in\mathcal {T}_{h}$, specially, $J(v)|_{E} :=(v|_{T})|_{E}$
if $E=\overline{T}\cap\partial\Omega$ . Furthermore, for
$T\in\mathcal{T}_{h}$, we denote by $\omega_{T}$ the union of all
elements in $\mathcal{T}_{h}$ sharing one edge with $T$, and define
the patch of $E\in\varepsilon_{h}$ by
\begin{equation*}
\omega_{E} :=\bigcup\{T\in\mathcal{T}_{h} : E\subset\overline{T}\}.
\end{equation*}
Denote $\Gamma_{h} :=\bigcup\varepsilon_{h}$, and let $J
:H^{1}(\bigcup\mathcal{T}_{h})\rightarrow L^{2}(\Gamma_{h})$ be an
operator with $H^{1}(\bigcup\mathcal{T}_{h}) :=\{v\in L^{2}(\Omega)
: \forall T\in\mathcal{T}_{h}, v|_{T}\in H^{1}(T)\}$.

Throughout the paper, the local versions of the differential
operators ${\rm div},\nabla, {\rm curl}$ are understood in the
distribution sense, i.e., in $D'(\Omega)$, namely, ${\rm
div}_{h},{\rm curl}_{h}:H^{1}(\bigcup\mathcal{T}_{h})^{2}\rightarrow
L^{2}(\Omega)$ and
$\nabla_{h}:H^{1}(\bigcup\mathcal{T}_{h})\rightarrow
L^{2}(\Omega)^{2}$ are defined such that, e.g., ${\rm div}_{h}v|_{T}
:={\rm div}(v|_{T})\ {\rm in}\ D'(T)$, for all $\ T\
\in\mathcal{T}_{h}$.

\subsection{A posteriori error estimators}
For all $E\in\varepsilon_{h}$, let $\tau$ be the unit tangential
vector along $E$, and $(p_{h},u_{h})\in M_{h}\times L_{h}$ be the
solution of (\ref{quasi-optimality 2.3}) with respect to the
triangulation $\mathcal{T}_{h}$. Then the local estimator is defined
by (see \cite{Carstensen})
\begin{equation*}
\begin{array}{lll}
\eta_{T,\kappa}^{2} :&=&||h^{\kappa}(f+{\rm div}\
p_{h})||_{L^{2}(T)}^{2}+h_{T}^{2}||{\rm
curl}(A^{-1}p_{h})||_{L^{2}(T)}^{2}\vspace{2mm}\\
&+&||h^{1/2}J(A^{-1}p_{h}\cdot\tau)||_{L^{2}(\partial T)}^{2}+
||h(A^{-1}p_{h}-\nabla_{h}u_{h})||_{L^{2}(T)}^{2}
\end{array}
\end{equation*}
with  $0\leq\kappa\leq1$ and the global estimator is given as
\begin{equation*}
\eta_{h,\kappa}^{2}
:=\sum\limits_{T\in\mathcal{T}_{h}}\eta_{T,\kappa}^{2}.
\end{equation*}
Here $ {\rm curl} \psi:=\frac{\partial\psi_{2}}{\partial
x_{1}}-\frac{\partial\psi_{1}}{\partial x_{2}} $ for
$\psi=(\psi_{1},\psi_{2})^{T}$.
For convenience we also define the stress variable error in weighted norm
$$\mathcal{E}_{h}^{2}
:=||A^{-1/2}(p-p_{h})||_{L^{2}(\Omega)}^{2}+||h{\rm
div}(p-p_{h})||_{L^{2}(\Omega)}^{2}.$$

Note that in this paper, the Curls of a scalar function $\phi$ are involved as
\begin{equation*}
{\rm Curl} \phi:=(-\frac{\partial\phi}{\partial
x_{2}},\frac{\partial\phi}{\partial x_{1}})^{T}.
\end{equation*}

In \cite{Carstensen0}, reliability of the a posteriori error estimator
$\eta_{h,\kappa}$ with estimates of the stress and displacement
variables in weighted norm was obtained  under a weak regularity assumption on $A$ (see Section 4.2 in
\cite{Carstensen0}),  whereas efficiency of $\eta_{h,\kappa}$ was derived by assuming additionally that   $A^{-1}p_{h}$ is a piecewise polynomial vector. In Section 3, we shall prove the efficiency without this additional
assumption on the coefficient matrix.  But
 we pay the price to introduce  oscillation of data.

In many applications the stress variable is of interest. We  define the local estimator for the stress variable error  as
\begin{equation*}
\begin{array}{lll}
\eta_{\mathcal{T}_{h}}^{2}(p_{h},T)
:&=&h_{T}^{2}||f-f_{h}||_{L^{2}(T)}^{2}+h_{T}^{2}||{\rm
curl}(A^{-1}p_{h})||_{L^{2}(T)}^{2}\vspace{2mm}\\
&+&h_{T}||J(A^{-1}p_{h}\cdot\tau)||_{L^{2}(\partial T)}^{2},
\end{array}
\end{equation*}
and define the global error estimator as
\begin{equation*}
\eta_{\mathcal{T}_{h}}^{2}(p_{h},\mathcal{T}_{h})
:=\sum\limits_{T\in\mathcal{T}_{h}}\eta_{\mathcal{T}_{h}}^{2}(p_{h},T).
\end{equation*}
We assume that, for a given triangulation $\mathcal{T}_{h}$ and a
pair of corresponding discrete solutions $(p_{h},u_{h})\in
M_{h}\times L_{h}$, the module $ESTIMATE$ for the stress variable
outputs the indicators
\begin{equation*}
\{\eta_{\mathcal{T}_{h}}^{2}(p_{h},T)\}_{T\in\mathcal{T}_{h}}=ESTIMATE(p_{h},\mathcal{T}_{h}).
\end{equation*}

Then the estimates of the stress variable error $\epsilon$ in a
weighted norm are reduced to (see \cite{Carstensen0})
\begin{equation}\label{Section 2.4}
||A^{-1/2}\epsilon||_{L^{2}(\Omega)}^{2}+||h{\rm div}\
\epsilon||_{L^{2}(\Omega)}^{2}\leq
C_{1}\eta_{\mathcal{T}_{h}}^{2}(p_{h},\mathcal{T}_{h}),
\end{equation}
where $C_{1}$ is a constant independent of the mesh size. In Section
3, we shall show  efficiency of the estimator
$\eta_{\mathcal{T}_{h}}(p_{h},\mathcal{T}_{h})$ for the stress
variable error in a weighted norm.

\subsection{Oscillation of data}
For an integer $n\geq l+1$, we denote by $\Pi_{n}^{2}$ the
$L^{2}-$best approximation operator onto the set of piecewise
polynomials of degree $\leq n$ over $T\in\mathcal{T}_{h}$ or
$E\in\varepsilon_{h}$,   denote by $id$ the identity operator, and
set $P_{n}^{2} :={\rm id}-\Pi_{n}^{2}$. We define  oscillation
$\widetilde{{\rm osc}}_{h}$ of   data as:
\begin{equation*}
\begin{array}{lll}
\widetilde{{\rm osc}}_{h}^{2} :&=&||hP_{n}^{2}{\rm
curl}(A^{-1}p_{h})||_{L^{2}(\Omega)}^{2}+
||h^{1/2}P_{n+1}^{2}J(A^{-1}p_{h}\cdot\tau)||_{L^{2}(\Gamma_{h})}^{2}\vspace{2mm}\\
&+& ||hP_{n}^{2}(A^{-1}p_{h}-\nabla_{h}u_{h})||_{L^{2}(\Omega)}^{2}.
\end{array}
\end{equation*}

For the stress variable error in a weighted norm, convergence and
quasi optimality of the AMFEM are involved in the oscillations of
the data including the right-hand side term $f$. Then we define the
oscillation of data as
\begin{equation*}
\begin{array}{lll}
{\rm osc}_{\mathcal{T}_{h}}^{2}(p_{h},T)
:&=&h_{T}^{2}||P_{n}^{2}{\rm curl}(A^{-1}p_{h})||_{L^{2}(T)}^{2}+
h_{T}||P_{n+1}^{2}J(A^{-1}p_{h}\cdot\tau)||_{L^{2}(\partial T)}^{2}\vspace{2mm}\\
&+& ||h(f-f_{h})||_{L^{2}(T)}^{2}\ \ \ {\rm for\ all}\
T\in\mathcal{T}_{h}.
\end{array}
\end{equation*}
Finally, for any subset $\mathcal{T}_{h}'\subset\mathcal{T}_{h}$, we
set
\begin{equation*}
{\rm osc}_{\mathcal{T}_{h}}^{2}(p_{h},\mathcal{T}_{h}')
:=\sum\limits_{T\in\mathcal{T}_{h}'}{\rm
osc}_{\mathcal{T}_{h}}^{2}(p_{h},T)\ \ {\rm and}\ \ {\rm
osc}_{h}^{2} :={\rm
osc}_{\mathcal{T}_{h}}^{2}(p_{h},\mathcal{T}_{h}).
\end{equation*}
We also define oscillation of $f$ as
\begin{equation*}
{\rm osc}^{2}(f,\mathcal{T}_{h})
:=||h(f-f_{h})||_{L^{2}(\Omega)}^{2}.
\end{equation*}

\begin{remark}\label{rem 2.1}
 For the estimator of the stress
variables, let $\mathcal{T}_{h}$ be a triangulation, $q_{h}\in
M_{h}$ be given. By substituting $p_{h}$ with $q_{h}$ in the
definitions of $\eta_{\mathcal{T}_{h}}(p_{h},T)$ and ${\rm
osc}_{\mathcal{T}_{h}}(p_{h},T)$, we can see that the indicator
$\eta_{\mathcal{T}_{h}}(q_{h},T)$ controls oscillation ${\rm
osc}_{\mathcal{T}_{h}}(q_{h},T)$, i.e., ${\rm
osc}_{\mathcal{T}_{h}}(q_{h},T)\leq\eta_{\mathcal{T}_{h}}(q_{h},T)$
for all $T\in\mathcal{T}_{h}$. In addition, for the stress
variables, the definitions of the error indicator and oscillation
are fully localized to $T$, which means there holds
$\eta_{\mathcal{T}_{H}}(q_{H},T)=\eta_{\mathcal{T}_{h}}(q_{H},T)$
and ${\rm osc}_{\mathcal{T}_{H}}(q_{H},T)={\rm
osc}_{\mathcal{T}_{h}}(q_{H},T)$ for any refinement
$\mathcal{T}_{h}$ of $\mathcal{T}_{H}$ with
$T\in\mathcal{T}_{h}\cap\mathcal{T}_{H}$ and $q_{H}\in M_{H}$.
Moreover, a combination of the monotonicity of local mesh sizes and
properties of the local $L^{2}-$projection yields
\begin{equation*}
\eta_{\mathcal{T}_{h}}(q_{H},\mathcal{T}_{h})\leq
\eta_{\mathcal{T}_{H}}(q_{H},\mathcal{T}_{H})\ \ {\rm and}\ \ {\rm
osc}_{\mathcal{T}_{h}}(q_{H},\mathcal{T}_{h})\leq{\rm
osc}_{\mathcal{T}_{H}}(q_{H},\mathcal{T}_{H}) \ \ \ \forall q_{H}\in
M_{H}.
\end{equation*}
\end{remark}

We note that in this paper, the triangulation $\mathcal{T}_{h}$
means a refinement of $\mathcal{T}_{H}$, all notations with respect
to the mesh $\mathcal{T}_{H}$ are defined similarly. Throughout the
rest of the paper we use the notation $A\lesssim B$ to represent
$A\leq CB$ with a mesh-size independent, generic constant $C>0$.
Moreover,  $A\approx B$ abbreviates $A\lesssim B\lesssim A$.

\subsection{The module MARK}
By relying on D\"{o}rfler marking, while only concerning the stress
variable error in a weighted norm, we select the elements to mark
according to the indicators for the stress variables, namely, given
a grid $\mathcal{T}_{H}$ with the set of indicators
$\{\eta_{\mathcal{T}_{H}}(p_{H},T)\}_{T\in\mathcal{T}_{H}}$ and
marking parameter $\theta\in(0,1]$, the module $MARK$ outputs a
subset of making elements $\mathcal{M}_{H}\subset\mathcal{T}_{H}$,
i.e.,
\begin{equation*}
\mathcal{M}_{H}=MARK(\{\eta_{\mathcal{T}_{H}}(p_{H},T)\}_{T\in
\mathcal{T}_{H}},\mathcal{T}_{H},\theta),
\end{equation*}
such that $\mathcal{M}_{H}$ satisfies D\"{o}rfler property
\begin{equation}\label{Dorfler property}
\eta_{\mathcal{T}_{H}}(p_{H},\mathcal{M}_{H})\geq
\theta\eta_{\mathcal{T}_{H}}(p_{H},\mathcal{T}_{H}).
\end{equation}
\subsection{The module REFINE}
In the $REFINE$ step, we suppose that the refinement rule, such as
the longest edge bisection \cite{Rivara1,Rivara2} and newest vertex
bisection \cite{Sewell,Mitchell1,Mitchell2}, is guaranteed to
produce conforming and shape regular mesh. Given a fixed integer
$b\geq1$, a mesh $\mathcal{T}_{H}$, and a subset
$\mathcal{M}_{H}\subset\mathcal{T}_{H}$ of marked elements, a
conforming triangulation $\mathcal{T}_{h}$ is output by
\begin{equation*}
\mathcal{T}_{h}=REFINE(\mathcal{T}_{H},\mathcal{M}_{H}),
\end{equation*}
where all elements of $\mathcal{M}_{H}$ are at least bisected $b$
times. Note that not only marked elements get refined but also
additional elements are refined to recovery the conformity of
triangulations. Let $\mathcal{R}
:=\mathcal{R}_{\mathcal{T}_{H}\rightarrow\mathcal{T}_{h}} :=
\mathcal{T}_{H}/(\mathcal{T}_{H}\cap\mathcal{T}_{h})$ denote the set
of refined elements, which means
$\mathcal{M}_{H}\subset\mathcal{R}_{\mathcal{T}_{H}\rightarrow\mathcal{T}_{h}}$.

In general, the number of these additionally refined elements is not
controlled by $\#\mathcal{M}_{H}$, that is to say,
$\#\mathcal{T}_{h}-\#\mathcal{T}_{H}$ cannot be bounded by
$C\mathcal{M}_{H}$ with a positive constant $C$, which is
independent of $\mathcal{T}_{H}$ and may depend on the refinement
level. On the other hand, by arguing with the entire sequence
$\{\mathcal{T}_{k}\}_{k\geq0}$ of refinement, Binev, Dahman, and
DeVore showed in two dimensions that the cumulative number of
elements added by insuring conformity does not inflate the total
number of marked elements \cite{Binev;Dahmen;DeVore}. Stevenson
generalized this result to higher dimensions \cite{Stevenson1}.

\begin{lemma}\label{lem 2.2}(\cite{Stevenson1}) \ (Complexity of $REFINE$).
 Assume that $\mathcal{T}_{0}$ verifies condition (b) of section
4 in \cite{Stevenson1}. Let $\{\mathcal{T}_{k}\}_{k\geq0}$ be any
conforming triangulation sequence refined from a shape regular
triangulation $\mathcal{T}_{0}$, where $\mathcal{T}_{k+1}$ is
generated from $\mathcal{T}_{k}$ by
$\mathcal{T}_{k+1}=REFINE(\mathcal{T}_{k},\mathcal{M}_{k})$ with a
subset $\mathcal{M}_{k}\subset\mathcal{T}_{k}$. Then there exists a
constant $C_{0}$ solely depending on $\mathcal{T}_{0}$ and $b$ such
that
\begin{equation*}
\#\mathcal{T}_{k}-\#\mathcal{T}_{0}\leq
C_{0}\sum\limits_{j=0}^{k-1}\#\mathcal{M}_{j}\ \ \ {\rm for\ all}\
k\geq1.
\end{equation*}
\end{lemma}

\subsection{Adaptive algorithm}
We now collect the modules described in the previous sections to
obtain the AMFEM of the stress variables. In doing this, we replace
the subscript $H$ (or $h$) by an iteration counter called $k\geq0$.
Let $\mathcal{T}_{0}$ be a shape regular triangulation, $\eta_{0}
:=\eta_{\mathcal{T}_{0}}(p_{0,\mathcal{T}_{0}})$ denote the error
indicator onto the initial mesh $\mathcal{T}_{0}$, with a right hand
side $f\in L^{2}(\Omega)$, a tolerance $\varepsilon$, and a
parameter $\theta\in(0,1]$. The basic loop of the AMFEM is then
given by the following iterations:\\
\begin{center}{\hbox{\hspace{37mm}\ \ \ \ \ \ \
The algorithm for AMFEM}} \scriptsize
\begin{tabular}{|c|} \hline
$[\mathcal{T}_{N},(p_{N},u_{N})]=AMFEM(\mathcal{T}_{0},f,\varepsilon,\theta)$\\
set $k =0,\eta_{k}=\eta_{0}$ and iterate\\
WHILE $\eta_{k}\geq\varepsilon$\ DO\\
\hspace{-24mm}(1) $(p_{k},u_{k})=SOLVE(\mathcal{T}_{k})$; \\
\hspace{-1mm}(2)
$\{\eta_{k}(p_{k},T)\}_{T\in\mathcal{T}_{k}}=ESTIMATE(p_{k},\mathcal{T}_{k})$;\\
\hspace{-4mm}(3)
$\mathcal{M}_{k}=MARK(\{\eta_{k}(p_{k},T)\}_{T\in\mathcal{T}_{k}},\mathcal{T}_{k},\theta)$;\\
\hspace{-3mm}(4)
$\mathcal{T}_{k+1}=REFINE(\mathcal{T}_{k},\mathcal{M}_{k})$;
$k=k+1$.\\
END WHILe\\
$\mathcal{T}_{N}=\mathcal{T}_{k}$. \\
END AMFEM\\ \hline
\end{tabular}
\end{center}

We note that the AMFEM for the stress variables is a standard
algorithm in which it employs only the error estimator
$\{\eta_{\mathcal{T}_{k}}(p_{k},T)\}_{T\in\mathcal{T}_{k}}$, does
not use the oscillation indicators $\{{\rm
osc}_{\mathcal{T}_{k}}(p_{k},T)\}_{T\in\mathcal{T}_{k}}$, and does
not need the interior node property for marked elements.

\setcounter{lemma}{0} \setcounter{theorem}{0}
\setcounter{section}{2} \setcounter{equation}{0}
\section{Analysis of efficiency for estimators}
We devote this section to the analysis of efficiency of the a
posteriori error estimator $\eta_{h,\kappa}$ for the stress and
displacement variables in a weighted norm. Herein, We avoid the
additional assumption that $A^{-1}p_{h}$ is a polynomial vector on each
element, which is necessary for the proof of the efficiency of
$\eta_{h,\kappa}$ in \cite{Carstensen0} for the RT, BDM, and BDFM
elements.

\begin{lemma}\label{lem 3.1}  Let $(p_{h},u_{h})\in M_{h}\times L_{h}$
be a pair of discrete solutions of (\ref{quasi-optimality 2.3}),
$P_{n}^{2}$ denote the operator defined in Section 2.5. Then, for
all $T\in\mathcal{T}_{h}$, it holds
\begin{equation}\label{Lemma 3.1}
h_{T}||{\rm
curl}(A^{-1}p_{h})||_{L^{2}(T)}\lesssim||A^{-1/2}\epsilon||_{L^{2}(T)}+
h_{T}||P_{n}^{2}{\rm curl}(A^{-1}p_{h})||_{L^{2}(T)}.
\end{equation}
\end{lemma}
\begin{proof} From the triangle inequality, we have
\begin{equation}\label{Lemma 3.1.1}
||{\rm curl}(A^{-1}p_{h})||_{L^{2}(T)}^{2}\leq2(||P_{n}^{2}{\rm
curl}(A^{-1}p_{h})||_{L^{2}(T)}^{2}+||\Pi_{n}^{2}{\rm
curl}(A^{-1}p_{h})||_{L^{2}(T)}^{2}).
\end{equation}
For all $T\in\mathcal{T}_{h}$, let $\psi_{T}$ denote the bubble
function on $T$ with zero boundary values on $T$ and
$0\leq\psi_{T}\leq1$, the equivalence of norms
$||\psi_{T}^{1/2}\cdot||_{L^{2}(T)}$ and $||\cdot||_{L^{2}(T)}$ for
polynomials implies
\begin{equation}\label{Lemma 3.1.2}
\begin{array}{lll}
||\Pi_{n}^{2}{\rm
curl}(A^{-1}p_{h})||_{L^{2}(T)}^{2}&\approx&||\psi_{T}^{1/2}\Pi_{n}^{2}{\rm
curl}(A^{-1}p_{h})||_{L^{2}(T)}^{2}\vspace{2mm}\\
&=&(\psi_{T}\Pi_{n}^{2}{\rm curl}(A^{-1}p_{h}),{\rm
curl}(A^{-1}p_{h}))_{0,T}\vspace{2mm}\\
&\ & -(\psi_{T}\Pi_{n}^{2}{\rm curl}(A^{-1}p_{h}),P_{n}^{2}{\rm
curl}(A^{-1}p_{h}))_{0,T}.
\end{array}
\end{equation}

From $\epsilon :=p-p_{h}$, $p :=A\nabla u$, Stokes theory, and
integration by parts, we have
\begin{equation}\label{Lemma 3.1.3}
\begin{array}{lll}
&\ &\displaystyle\int_{T}\psi_{T}\Pi_{n}^{2}{\rm
curl}(A^{-1}p_{h})\cdot{\rm
curl}(A^{-1}p_{h})\\
&=&\displaystyle\int_{T}\psi_{T}\Pi_{n}^{2}{\rm
curl}(A^{-1}p_{h})\cdot{\rm curl}(\nabla
u-A^{-1}\epsilon)\vspace{2mm}\\
&=&\displaystyle-\int_{T}\psi_{T}\Pi_{n}^{2}{\rm
curl}(A^{-1}p_{h})\cdot{\rm
curl}(A^{-1}\epsilon)\vspace{2mm}\\
&=&\displaystyle\int_{T}{\rm Curl}(\psi_{T}\Pi_{n}^{2}{\rm
curl}(A^{-1}p_{h}))\cdot(A^{-1}\epsilon)\vspace{2mm}\\
&\leq&\displaystyle|\psi_{T}\Pi_{n}^{2}{\rm
curl}(A^{-1}p_{h})|_{H^{1}(T)}||A^{-1}\epsilon||_{L^{2}(T)}.
\end{array}
\end{equation}

A combination of (\ref{Lemma 3.1.1})-(\ref{Lemma 3.1.3}), together
with an inverse estimation and the property of $L^{2}-$projection,
yields
\begin{equation*}
\begin{array}{lll}
||{\rm curl}(A^{-1}p_{h})||_{L^{2}(T)}^{2}&\lesssim&(||P_{n}^{2}{\rm
curl}(A^{-1}p_{h})||_{L^{2}(T)}+h_{T}^{-1}||A^{-1/2}\epsilon||_{L^{2}(T)})
\vspace{2mm}\\
&\ & \times||{\rm curl}(A^{-1}p_{h})||_{L^{2}(T)}.
\end{array}
\end{equation*}
This implies the desired result.
\end{proof}
\begin{lemma}\label{lem 3.2}  Let $(p_{h},u_{h})\in M_{h}\times L_{h}$
be the discrete solutions of (\ref{quasi-optimality 2.3}),
$P_{n+1}^{2}$ denote the operator defined in Section 2.5. Then, for
all $E\in\varepsilon_{h}$, it holds
\begin{equation}\label{Lemma 3.2}
\begin{array}{lll}
h_{E}^{1/2}||J(A^{-1}p_{h}\cdot\tau)||_{L^{2}(E)}&\lesssim&
h_{E}^{1/2}||P_{n+1}^{2}J(A^{-1}p_{h}\cdot\tau)||_{L^{2}(E)}\vspace{2mm}\\
&\ & +h_{E}||P_{n}^{2}{\rm
curl}_{h}(A^{-1}p_{h})||_{L^{2}(\omega_{E})}+||A^{-1/2}\epsilon||_{L^{2}(\omega_{E})}.
\end{array}
\end{equation}
\end{lemma}
\begin{proof} For all
$E\in\varepsilon_{h}$, let $\psi_{E}$ denote the bubble function on
$E$ with the support set $\omega_{E}$ and $0\leq\psi_{E}\leq1$. Put
$\sigma :=J(A^{-1}p_{h}\cdot\tau)$. since
$\Pi_{n+1}^{2}:L^{2}(E)\rightarrow P_{n+1}(E)$ is an
$L^{2}-$projection operator, where $P_{n+1}(E)$ is a set of
polynomials of total degree $\leq n+1$ over $E$, there exists an
extension operator $P:C(E)\rightarrow C(\omega_{E})$
\cite{Verfurth-1,Verfurth-2} such that
\begin{equation}\label{Lemma 3.2.1}
P\Pi_{n+1}^{2}\sigma|_{E}=\Pi_{n+1}^{2}\sigma\ \ \ {\rm and}\ \ \
||\psi_{E}P\Pi_{n+1}^{2}\sigma||_{L^{2}(\omega_{E})}\approx
h_{E}^{1/2}||\Pi_{n+1}^{2}\sigma||_{L^(E)}.
\end{equation}
From the triangle inequality, we obtain
\begin{equation}\label{Lemma 3.2.2}
||\sigma||_{L^{2}(E)}^{2}\leq2(||P_{n+1}^{2}\sigma||_{L^{2}(E)}^{2}+
||\Pi_{n+1}^{2}\sigma||_{L^{2}(E)}^{2}).
\end{equation}
The equivalence of norms $||\psi_{E}^{1/2}\cdot||_{L^{2}(E)}$ and
$||\cdot||_{L^{2}(E)}$ for polynomials implies
\begin{equation}\label{Lemma 3.2.3}
\begin{array}{lll}
||\Pi_{n+1}^{2}\sigma||_{L^{2}(E)}^{2}&\approx&||\psi_{E}^{1/2}
\Pi_{n+1}^{2}\sigma||_{L^{2}(E)}^{2}
=(\psi_{E}\Pi_{n+1}^{2}\sigma,\Pi_{n+1}^{2}\sigma)_{0,E}\vspace{2mm}\\
&=&(\psi_{E}\Pi_{n+1}^{2}\sigma,\sigma)_{0,E}-
(\psi_{E}\Pi_{n+1}^{2}\sigma,P_{n+1}^{2}\sigma)_{0,E}.
\end{array}
\end{equation}
From integration by parts, we get
\begin{equation}\label{Lemma 3.2.4}
\begin{array}{lll}
(\psi_{E}\Pi_{n+1}^{2}\sigma,\sigma)_{0,E}&=&\displaystyle\int_{E}
\psi_{E}P\Pi_{n+1}^{2}\sigma\cdot\sigma
=\displaystyle\int_{\omega_{E}}A^{-1}p_{h}\cdot{\rm
curl}(\psi_{E}P\Pi_{n+1}^{2}\sigma)\vspace{2mm}\\
&+&\displaystyle\int_{\omega_{E}}{\rm
curl}_{h}(A^{-1}p_{h})\psi_{E}P\Pi_{n+1}^{2}\sigma\vspace{2mm}\\
\end{array}
\end{equation}
From Stokes theory, and noticing that $A^{-1}p_{h}=\nabla
u-A^{-1}\epsilon$, we have
\begin{equation}\label{Lemma 3.2.5}
\displaystyle\int_{\omega_{E}}A^{-1}p_{h}\cdot{\rm
curl}(\psi_{E}P\Pi_{n+1}^{2}\sigma)=-\int_{\omega_{E}}A^{-1}\epsilon\cdot{\rm
curl}(\psi_{E}P\Pi_{n+1}^{2}\sigma).
\end{equation}

A combination of (\ref{Lemma 3.2.4}) and (\ref{Lemma 3.2.5}) yields
\begin{equation}\label{Lemma 3.2.6}
\begin{array}{lll}
(\psi_{E}\Pi_{n+1}^{2}\sigma,\sigma)_{0,E}&\leq&
||A^{-1}\epsilon||_{L^{2}(\omega_{E})}|\psi_{E}P\Pi_{n+1}^{2}
\sigma|_{H^{1}(\omega_{E})}\vspace{2mm}\\
&\ & +||{\rm
curl}_{h}(A^{-1}p_{h})||_{L^{2}(\omega_{E})}||\psi_{E}P\Pi_{n+1}^{2}
\sigma||_{L^{2}(\omega_{E})}\vspace{2mm}\\
\end{array}
\end{equation}
This inequality, together with an inverse estimate and (\ref{Lemma 3.2.1}), yields
\begin{equation}\label{Lemma 3.2.7}
\begin{array}{lll}
(\psi_{E}\Pi_{n+1}^{2}\sigma,\sigma)_{0,E}&\lesssim&
(||A^{-1}\epsilon||_{L^{2}(\omega_{E})}+
h_{E}||{\rm curl}_{h}(A^{-1}p_{h})||_{L^{2}(\omega_{E})})\vspace{2mm}\\
&\ & \times h_{E}^{-1/2}||\Pi_{n+1}^{2}\sigma||_{L^{2}(E)}.
\end{array}
\end{equation}
We apply Lemma 
\ref{lem 3.1}   to the above inequality (\ref{Lemma 3.2.7}) and use
  the property of $L^{2}-$projection to get
\begin{equation}\label{Lemma 3.2.8}
(\psi_{E}\Pi_{n+1}^{2}\sigma,\sigma)_{0,E}\lesssim(h_{E}||P_{n}^{2}{\rm
curl}_{h}(A^{-1}p_{h})||_{L^{2}(\omega_{E})}+
||A^{-\frac{1}{2}}\epsilon||_{L^{2}(\omega_{E})})h_{E}^{-\frac{1}{2}}||\sigma||_{L^{2}(E)}.
\end{equation}
From (\ref{Lemma 3.2.2}), (\ref{Lemma 3.2.3}), the property of
$L^{2}-$projection,  equivalence of norms
$||\psi_{E}^{1/2}\cdot||_{L^{2}(E)}$ and $||\cdot||_{L^{2}(E)}$ for
polynomials, and (\ref{Lemma 3.2.8}), we obtain
\begin{equation}\label{Lemma 3.2.9}
\begin{array}{lll}
h_{E}^{1/2}||\sigma||_{L^{2}(E)}^{2}&\lesssim&
(h_{E}^{1/2}||P_{n+1}^{2}\sigma||_{L^{2}(E)}+h_{E}||P_{n}^{2}{\rm
curl}_{h}(A^{-1}p_{h})||_{L^{2}(\omega_{E})}\vspace{2mm}\\
&\ &
+||A^{-1/2}\epsilon||_{L^{2}(\omega_{E})})||\sigma||_{L^{2}(E)}.
\end{array}
\end{equation}
The above inequality (\ref{Lemma 3.2.9}) implies the desired result
(\ref{Lemma 3.2}) by canceling one $||\sigma||_{L^{2}(E)}$.
\end{proof}

\begin{lemma}\label{lem 3.3}   Let $(p_{h},u_{h})\in M_{h}\times L_{h}$
be a pair of discrete solutions of (\ref{quasi-optimality 2.3}),
$P_{n}^{2}$ denote the operator defined in Section 2.5. Then, for
all $T\in\mathcal{T}_{h}$, it holds
\begin{equation}\label{Lemma 3.3}
\begin{array}{lll}
h_{T}||A^{-1}p_{h}-\nabla_{h}u_{h}||_{L^{2}(T)}&\lesssim&
h_{T}||P_{n}^{2}(A^{-1}p_{h}-\nabla_{h}u_{h})||_{L^{2}(T)}\vspace{2mm}\\
&\ & +||e||_{L^{2}(T)}+||A^{-1/2}\epsilon||_{L^{2}(T)}.
\end{array}
\end{equation}
\end{lemma}
\begin{proof} Denote $\mu :=A^{-1}p_{h}-\nabla_{h}u_{h}$, and let
$\psi_{T}$ be the bubble function defined in the proof of Lemma 
ef{lem 3.1}.
Since $\Pi_{n}^{2}$ is the $L^{2}-$best approximation operator onto the set
of polynomials of degree $\leq n$ over $T\in\mathcal{T}_{h}$, we have
\begin{equation}\label{Lemma 3.3.1}
||\mu||_{L^{2}(T)}^{2}\leq2(||P_{n}^{2}\mu||_{L^{2}(T)}||\mu||_{L^{2}(T)}
+||\Pi_{n}^{2}\mu||_{L^{2}(T)}^{2}).
\end{equation}

The property of the bubble function $\psi_{T}$ indicates
\begin{equation}\label{Lemma 3.3.2}
\begin{array}{lll}
||\Pi_{n}^{2}\mu||_{L^{2}(T)}^{2}&\approx&
||\psi_{T}^{1/2}\Pi_{n}^{2}\mu||_{L^{2}(T)}^{2}=
(\psi_{T}\Pi_{n}^{2}\mu,\Pi_{n}^{2}\mu)_{0,T}\vspace{2mm}\\
&=&(\psi_{T}\Pi_{n}^{2}\mu,\mu)_{0,T}-(\psi_{T}\Pi_{n}^{2}\mu,P_{n}^{2}\mu)_{0,T}.
\end{array}
\end{equation}
Since $\epsilon =p-p_{h}$, $e =u-u_{h}$, integration by parts and an
inverse estimate lead to
\begin{equation}\label{Lemma 3.3.3}
\begin{array}{lll}
\displaystyle\int_{T}\psi_{T}\Pi_{n}^{2}\mu\cdot\mu&=&\displaystyle\int_{T}(\nabla_{h}e-A^{-1}\epsilon)
\cdot\psi_{T}\Pi_{n}^{2}\mu\vspace{2mm}\\
&=&\displaystyle-\int_{T}{\rm div}(\psi_{T}\Pi_{n}^{2}\mu)e-
\int_{T}A^{-1}\epsilon\cdot\psi_{T}\Pi_{n}^{2}\mu\vspace{2mm}\\
&\leq&\displaystyle||e||_{L^{2}(T)}|\psi_{T}\Pi_{n}^{2}\mu|_{H^{1}(T)}+
||A^{-1}\epsilon||_{L^{2}(T)}||\psi_{T}\Pi_{n}^{2}\mu||_{L^{2}(T)}\vspace{2mm}\\
&\lesssim&\displaystyle(h_{T}^{-1}||e||_{L^{2}(T)}+
||A^{-1/2}\epsilon||_{L^{2}(T)})||\mu||_{L^{2}(T)}.
\end{array}
\end{equation}
A combination of (\ref{Lemma 3.3.1})-(\ref{Lemma 3.3.3}) yields
\begin{equation}\label{Lemma 3.3.4}
h_{T}||\mu||_{L^{2}(T)}^{2}\lesssim(h_{T}||P_{n}^{2}\mu||_{L^{2}(T)}+
||e||_{L^{2}(T)}+||A^{-1/2}\epsilon||_{L^{2}(T)})||\mu||_{L^{2}(T)}.
\end{equation}
The assertion (\ref{Lemma 3.3}) then follows from the above
inequality (\ref{Lemma 3.3.4}).
\end{proof}

We now prove efficiency of the estimator $\eta_{h,\kappa}$ by
using the above three lemmas.

\begin{theorem}\label{thm 3.4}  Let $(p,u)$ and $(p_{h},u_{h})\in M_{h}\times L_{h}$ be the
solutions of (\ref{quasi-optimality 2.1}) and (\ref{quasi-optimality
2.3}), respectively, and $\eta_{h,\kappa}$ and $\widetilde{{\rm
osc}}_{h}$ be  defined as in Section 2.4 and 2.5. Then, for the
estimator of the stress and displacement variables for the RT, BDM,
and BDFM elements, there exists a constant hidden in $\lesssim$,
independent of mesh-size, such that
\begin{equation}\label{Lemma 3.4}
\eta_{h,\kappa}^{2}\lesssim||A^{-1/2}(p-p_{h})||_{L^{2}(\Omega)}^{2}+
||h^{\kappa}{\rm
div}(p-p_{h})||_{L^{2}(\Omega)}^{2}+||u-u_{h}||_{L^{2}(\Omega)}^{2}+\widetilde{{\rm
osc}}_{h}^{2}.
\end{equation}
\end{theorem}
\begin{proof} Notice that for all $T\in\mathcal{T}_{h}$, it holds
\begin{equation*}
||h^{\kappa}(f+{\rm div}\ p_{h})||_{L^{2}(T)}=||h^{\kappa}{\rm div}
\ \epsilon||_{L^{2}(T)}\ \ \ {\rm for\ all}\ \ \ 0\leq\kappa\leq1.
\end{equation*}
A combination of Lemmas 3.1-3.3 yields the
assertion (\ref{Lemma 3.4}) by summing over all
$T\in\mathcal{T}_{h}$ and $E\in\varepsilon_{h}$.
\end{proof}

\begin{theorem}\label{thm 3.5}  Under the assumptions of Theorem 
\ref{thm 3.4},
let $\eta_{\mathcal{T}_{h}}(p_{h},\mathcal{T}_{h})$, $\mathcal{E}_h$, and ${\rm osc}_{h}$ be  defined as in Section 2.4 and
2.5. Then, for the estimator of the stress variables for the RT,
BDM, and BDFM elements, there exists a constant $C_{2}$ independent
of mesh-size, such that
\begin{equation}\label{Lemma 3.5}
C_{2}\eta_{\mathcal{T}_{h}}(p_{h},\mathcal{T}_{h})^{2}\leq
\mathcal{E}_h^{2}+{\rm osc}_{h}^{2}.
\end{equation}
\end{theorem}
\begin{proof}
Since $-{\rm div}\ \epsilon=f+{\rm div}\ p_{h}$, combining Lemmas \ref{lem 3.1}-\ref{lem 3.2}, and summing over all $T\in\mathcal{T}_{h}$ and
$E\in\varepsilon_{h}$, we obtain the desired result (\ref{Lemma
3.5}).
\end{proof}

\section{Auxiliary results}
In this section, we will give some auxiliary results for 
convergence and quasi-optimality of the AMFEM for the stress
variable.

\subsection{Quasi-orthogonality}
\begin{lemma}\label{lem 4.1}  Given a function $f\in L_{0}^{2}(\Omega)$,
there exists a function $q\in H_{0}^{1}(\Omega)^{2}$ such that
\begin{equation*}
{\rm div}\ q=f\ \ \ \ {\rm and}\ \ \ \
||q||_{H^{1}(\Omega)}\leq||f||_{L^{2}(\Omega)}.
\end{equation*}
\end{lemma}

We refer to \cite{Brenner;Scott1,Arnold,Duran} for  detailed proofs
of this lemma respectively on smooth or convex, non-convex and general Lipschitz
domains.

\begin{lemma}\label{lem 4.2}  Let $L_{h}$, $M_h$ be respectively the discrete displacement and stress spaces given in Section 2.2.  Set $W :=H({\rm div},\Omega)\cap
L^{\varrho}(\Omega)^{2}$ for some $\varrho>2$, and let $\Pi_{L_{h}}$, $id$ and $\perp$ be respectively the
$L^{2}(\Omega)-$projection onto   $L_{h}$,  
the identity operator  and 
$L^{2}(\Omega)-$orthogonality. Then there exists an operator
$\Pi_{h}:W\rightarrow M_{h}$ with the following commuting diagram
\begin{equation}\label{Lemma 4.2.1}
\begin{array}{lll}
&W&\stackrel{{\rm div}}{\longrightarrow}\ \  L^{2}(\Omega)\\
&\downarrow\Pi_{h}&\ \ \ \ \ \ \ \ \ \ \downarrow\Pi_{L_{h}}\\
&M_{h}&\stackrel{{\rm div}}{\longrightarrow}\ \ \ L_{h}
\end{array}
\end{equation}
such that

\noindent
(${I}$) it holds a local estimate (note that
$H^{1}(\bigcup\mathcal{T}_{h})^{2}\cap H({\rm div},\Omega)\subset
W$)
\begin{equation}\label{Lemma 4.2.2}
||h^{-1}(id-\Pi_{h})q||_{L^{2}(\Omega)}\lesssim
|q|_{H^{1}(\bigcup\mathcal{T}_{h})}\ \ {\rm for\ all}\ \ q\in
H^{1}(\bigcup\mathcal{T}_{h})^{2}\cap H({\rm div},\Omega);
\end{equation}

\noindent($II$) $\Pi_{h}$ approximates the normal components on element edges
with
\begin{equation*}
\displaystyle\int_{E}v_{h}(id-\Pi_{h})q\cdot\nu_{E}ds=0\ \ \ {\rm
for\ all}\ \ E\in\varepsilon_{h}, v_{h}\in L_{h}, q\in W,
\end{equation*}
where $\nu_{E}$ is the unit normal vector along $E$.
\end{lemma}

For the detailed construction of such interpolation operator
$\Pi_{h}$ and proof of these properties, we refer to
\cite{Hiptmair,Arnold0,Brezzi;Fortin}. Note that the above commuting
diagram means
\begin{equation*}
{\rm div}(id-\Pi_{h})W \perp L_{h}.
\end{equation*}

\begin{lemma}\label{lem 4.3}  Let $\mathcal{T}_{h}$ and
$\mathcal{T}_{H}$ be two nested triangulations, $\Pi_{L_{H}}$ be the
$L^{2}(\Omega)$ $-$projection onto $L_{H}$, and $(p_{h},u_{h})\in
M_{h}\times L_{h}$ be the solutions of (\ref{quasi-optimality 2.3}).
Then for any $T\in\mathcal{T}_{H}$, there exists a positive constant
$C_{0}$ depending only on the shape regularity of $\mathcal{T}_{H}$,
such that
\begin{equation}\label{Lemma 4.2.4}
||u_{h}-\Pi_{L_{H}}u_{h}||_{L^{2}(T)}\leq\sqrt{C_{0}}H_{T}||A^{-1/2}p_{h}||_{L^{2}(T)}.
\end{equation}
\end{lemma}
\begin{proof} Let $\Pi_{L_{h}}$ denote the $L^{2}-$projection operator over $L_{h}$.
 For any $T\in\mathcal{T}_{H}$, by the
definition of $L^{2}-$projection operator $\Pi_{L_{H}}$, we have
$\int_{T}(\Pi_{L_{h}}-\Pi_{L_{H}})u_{h}=0$, i.e.,
$(\Pi_{L_{h}}-\Pi_{L_{H}})u_{h}\in L_{0}^{2}(T)$. We thus can apply
Lemma \ref{lem 4.1} to find a function $q\in H_{0}^{1}(T)^{2}$ such that
\begin{equation}\label{Lemma 4.2.5}
{\rm div}\ q=(\Pi_{L_{h}}-\Pi_{L_{H}})u_{h}\ {\rm in}\ T\ {\rm and}\
||q||_{H^{1}(T)}\lesssim||(\Pi_{L_{h}}-\Pi_{L_{H}})u_{h}||_{L^{2}(T)}.
\end{equation}

We extend $q$ to $H_{0}^{1}(\Omega)^{2}$ by zero, since $\Pi_{h}$
(or $\Pi_{H}$) approximates the normal components on elements edge,
the second result (${\it II}$) of Lemma 
\ref{lem 4.2} implies that
$(\Pi_{h}-\Pi_{H})q\in M_{h}$ and ${\rm
supp}(\Pi_{h}-\Pi_{H})q\subseteq T$. 
 Noticing that ${\rm div}\ q\in L_{h}$, we have
\begin{equation}\label{Lemma 4.3.5}
({\rm div}\ q-(\Pi_{L_{h}}-\Pi_{L_{H}}){\rm div}\
q,(\Pi_{L_{h}}-\Pi_{L_{H}})u_{h})_{0,T}=0
\end{equation}
and
\begin{equation}\label{Lemma 4.3.6}
(u_{h}-(\Pi_{L_{h}}-\Pi_{L_{H}})u_{h},(\Pi_{L_{h}}-\Pi_{L_{H}}){\rm
div}\ q)_{0,T}=0.
\end{equation}
Since
$\Pi_{L_{h}}u_{h}=u_{h}$,
a combination of the first equality of (\ref{Lemma 4.2.5}),
and (\ref{Lemma 4.3.5})-(\ref{Lemma 4.3.6}) yields
\begin{eqnarray}\label{Lemma 4.3.7}
||u_h-\Pi_{L_{H}}u_{h}||_{L^{2}(T)}^{2}&=&||(\Pi_{L_{h}}-\Pi_{L_{H}})u_{h}||_{L^{2}(T)}^{2}\nonumber\\
&=&
((\Pi_{L_{h}}-\Pi_{L_{H}})u_{h},{\rm div}\
q)_{0,T}\nonumber\\
&=&(u_{h},(\Pi_{L_{h}}-\Pi_{L_{H}}){\rm div}\ q)_{0,T}.
\end{eqnarray}
Using the locality of $q$, the commuting property
(\ref{Lemma 4.2.1}) and (\ref{quasi-optimality 2.3}), we obtain
\begin{eqnarray}\label{Lemma 4.3.8}
(u_{h},(\Pi_{L_{h}}-\Pi_{L_{H}}){\rm div}\
q)_{0,T}&=&(u_{h},(\Pi_{L_{h}}-\Pi_{L_{H}}){\rm div}\
q)_{0,\Omega}\nonumber\\
&=&(u_{h},{\rm div}(\Pi_{h}-\Pi_{H})q)_{0,\Omega}\nonumber\\
&=&
-(A^{-1}p_{h},(\Pi_{h}-\Pi_{H})q)_{0,\Omega}\nonumber\\
&=&-(A^{-1}p_{h},(\Pi_{h}-\Pi_{H})q)_{0,T}.
\end{eqnarray}
The local approximation (\ref{Lemma 4.2.2}) of Lemma 
\ref{lem 4.2} indicates
{\small
\begin{eqnarray}\label{Lemma 4.3.10}
 |(A^{-1}p_{h},(\Pi_{h}-\Pi_{H})q)_{0,T}|&\leq&||A^{-1}p_{h}||_{L^{2}(T)}
(||q-\Pi_{h}q||_{L^{2}(T)} +||q-\Pi_{H}q||_{L^{2}(T)}) \nonumber\\
&\leq&
C_{0}H_{T}||A^{-1/2}p_{h}||_{L^{2}(T)}||q||_{H^{1}(T)}.
\end{eqnarray}
}
Finally, the desired result
(\ref{Lemma 4.2.4}) follows from   the second inequality of (\ref{Lemma 4.2.5}) and
(\ref{Lemma 4.3.7}) - (\ref{Lemma 4.3.10}).
\end{proof}

In order to prove the quasi-orthogonality, we need to introduce a
pair of auxiliary solutions. Let $f_{H} :=\Pi_{L_{H}}f$
denote the $L^{2}-$projection of $f$ over $L_{H}$, and  consider the following problem:
Find $(\widetilde{p}_{h},\widetilde{u}_{h})\in M_{h}\times L_{h}$
such that
\begin{equation}\label{quasi-optimality 2.4}
\begin{array}{lll}
\displaystyle(A^{-1}\widetilde{p}_{h},q_{h})_{0,\Omega}+({\rm div}\
q_{h},\widetilde{u}_{h})_{0,\Omega}=0\ \ &{\rm for\
all}&\displaystyle\
q_{h}\in M_{h},\vspace{2mm}\\
\displaystyle({\rm div}\ \widetilde{p}_{h},
v_{h})_{0,\Omega}=-(f_{H},v_{h})_{0,\Omega}\ \ &{\rm for\
all}&\displaystyle\ v_{h}\in L_{h}.
\end{array}
\end{equation}

In fact, the solution $(\widetilde{p}_{h},\widetilde{u}_{h})$ of
this auxiliary problem  may be
regarded as another approximation to the flux and displacement $(p,u)$.

\begin{lemma}\label{lem 4.4}  Let $\mathcal{T}_{h}$ and
$\mathcal{T}_{H}$ be two nested triangulations, ${\rm
osc}(f_{h},\mathcal{T}_{H})$ denote the oscillation of  $f_{h}
:=\Pi_{L_{h}}f$ over $\mathcal{T}_{H}$, $(p_{h},u_{h})$ and
$(\widetilde{p}_{h},\widetilde{u}_{h})$ be the solutions of
(\ref{quasi-optimality 2.3}) and (\ref{quasi-optimality 2.4}),
respectively. Then there exists a constant $C_{0}$ depending only on
the shape regularity of $\mathcal{T}_{H}$ such that
\begin{equation}\label{Lemma 4.4}
||A^{-1/2}(p_{h}-\widetilde{p}_{h})||_{L^{2}(\Omega)}\leq\sqrt{C_{0}}{\rm
osc}(f_{h},\mathcal{T}_{H}.)
\end{equation}
\end{lemma}
\begin{proof} Recall that $(p_{h}-\widetilde{p}_{h},u_{h}-
\widetilde{u}_{h})\in M_{h}\times L_{h}$ satisfies the equations
\begin{equation}\label{Lemma 4.4.1}
\begin{array}{lll}
\displaystyle(A^{-1}(p_{h}-\widetilde{p}_{h}),q_{h})_{0,\Omega}+({\rm
div}\ q_{h},u_{h}-\widetilde{u}_{h})_{0,\Omega}=0\ \ &{\rm for\
all}&\displaystyle\
q_{h}\in M_{h},\vspace{2mm}\\
\displaystyle({\rm div}(p_{h}-\widetilde{p}_{h},
v_{h})_{0,\Omega}=-(f_{h}-f_{H},v_{h})_{0,\Omega}\ \ &{\rm for\
all}&\displaystyle\ v_{h}\in L_{h}.
\end{array}
\end{equation}
According to the above equations (\ref{Lemma 4.4.1}), we choose
$q_{h}=p_{h}-\widetilde{p}_{h}$ and $v_{h}=u_{h}-\widetilde{u}_{h}$
to obtain
\begin{equation}\label{Lemma 4.4.2}
\begin{array}{lll}
||A^{-1/2}(p_{h}-\widetilde{p}_{h})||_{L^{2}(\Omega)}^{2}&=&-({\rm
div}(p_{h}-\widetilde{p}_{h}),u_{h}-\widetilde{u}_{h})_{0,\Omega}\vspace{2mm}\\
&=&(f_{h}-\Pi_{L_{H}}f,u_{h}-\widetilde{u}_{h})_{0,\Omega}=
(f_{h}-\Pi_{L_{H}}f,v_{h})_{0,\Omega}.
\end{array}
\end{equation}
Since $L_{H}\subset L_{h}$, it holds
\begin{equation*}
(\Pi_{L_{h}}f,v_{H})_{0,\Omega}=(f,v_{H})_{0,\Omega}=
(\Pi_{L_{H}}\Pi_{L_{h}}f,v_{H})_{0,\Omega}
=(\Pi_{L_{H}}f,v_{H})
\end{equation*}
for   all  $v_{H}\in L_{H}$. This implies 
$$\Pi_{L_{H}}\Pi_{L_{h}}f=\Pi_{L_{H}}f\ \ \text{ and }\ 
(\Pi_{L_{h}}f-\Pi_{L_{H}}f,\Pi_{L_{H}}v_{h})_{0,\Omega}=0.$$ 
From
(\ref{Lemma 4.4.2}), we have
\begin{equation}\label{4.14}
\begin{array}{lll}
||A^{-1/2}(p_{h}-\widetilde{p}_{h})||_{L^{2}(\Omega)}^{2}&=&
((\Pi_{L_{h}}-\Pi_{L_{H}})f,v_{h}-\Pi_{L_{H}}v_{h})_{0,\Omega} \vspace{2mm}\\
&=&\displaystyle\sum\limits_{T\in\mathcal{T}_{H}}((\Pi_{L_{h}}-
\Pi_{L_{H}})f,v_{h}-\Pi_{L_{H}}v_{h})_{0,T}.
\end{array}
\end{equation}
We apply Lemma 
\ref{lem 4.3} to $v_{h}-\Pi_{L_{H}}v_{h}$ in (\ref{4.14}) and obtain
\begin{equation*}
\begin{array}{lll}
||A^{-\frac{1}{2}}(p_{h}-\widetilde{p}_{h})||_{L^{2}(\Omega)}^{2}&\leq&\displaystyle
C_{0}^{\frac{1}{2}}\sum\limits_{T\in\mathcal{T}_{H}}H_{T}||f_{h}-
f_{H}||_{L^{2}(T)}||A^{-\frac{1}{2}}(p_{h}-\widetilde{p}_{h})||_{L^{2}(T)}\vspace{2mm}\\
&\leq&\sqrt{C_{0}}{\rm
osc}(f_{h},\mathcal{T}_{H})||A^{-1/2}(p_{h}-\widetilde{p}_{h})||_{L^{2}(\Omega)},
\end{array}
\end{equation*}
which leads to the desired result (\ref{Lemma 4.4}).
\end{proof}

We  state the property of
quasi-orthogonality as follows.

\begin{theorem}\label{thm 4.5}({\rm Quasi-orthogonality})  Given $f\in
L^{2}(\Omega)$ and two nested triangulations $\mathcal{T}_{h}$ and
$\mathcal{T}_{H}$, let $(p_{h},u_{h})$ and $(p_{H},u_{H})$ be the
solutions of (\ref{quasi-optimality 2.3}) with respect to
$\mathcal{T}_{h}$ and $\mathcal{T}_{H}$, respectively. Then it holds
\begin{equation}\label{Lemma 4.5}
(A^{-1}(p-p_{h}),p_{h}-p_{H})_{0,\Omega}\leq
C_{0}^{1/2}||A^{-1/2}(p-p_{h})||_{L^{2}(\Omega)}{\rm
osc}(f_{h},\mathcal{T}_{H}).
\end{equation}
Furthermore, for any $\delta_{1}>0$, it holds
\begin{equation}\label{Lemma 4.5.1}
\begin{array}{lll}
&&(1-\delta_{1})||A^{-1/2}(p-p_{h})||_{L^{2}(\Omega)}^{2}\\
&\leq&
||A^{-1/2}(p-p_{H})||_{L^{2}(\Omega)}^{2}
-||A^{-1/2}(p_{h}-p_{H})||_{L^{2}(\Omega)}^{2}
+\frac{C_{0}}{\delta_{1}}{\rm osc}^{2}(f_{h},\mathcal{T}_{H}).
\end{array}
\end{equation}
In particular, if ${\rm osc}(f_{h},\mathcal{T}_{H})=0$, then it
holds
\begin{equation}\label{Lemma 4.5.2}
||A^{-1/2}(p-p_{h})||_{L^{2}(\Omega)}^{2}=||A^{-1/2}(p-p_{H})||_{L^{2}(\Omega)}^{2}-
||A^{-1/2}(p_{h}-p_{H})||_{L^{2}(\Omega)}^{2}.
\end{equation}
\end{theorem}
\begin{proof} Let $(\widetilde{p}_{h},\widetilde{u}_{h})$ solve the problem
(\ref{quasi-optimality 2.4}), then we have
\begin{equation}\label{Lemma 4.5.3}
\begin{array}{lll}
(A^{-1}(p-p_{h}),\widetilde{p}_{h}-p_{H})_{0,\Omega}&=&-({\rm
div}(\widetilde{p}_{h}-p_{H}),u-\widetilde{u}_{h})_{0,\Omega}\vspace{2mm}\\
&=&(f_{H}-f_{H},u-\widetilde{u}_{h})_{0,\Omega}=0.
\end{array}
\end{equation}
From the above identity (\ref{Lemma 4.5.3}) and Lemma 
\ref{lem 4.4}, we obtain
\begin{equation}\label{Lemma 4.5.4}
\begin{array}{lll}
(A^{-1}(p-p_{h}),p_{h}-p_{H})_{0,\Omega}
&=&(A^{-1}(p-p_{h}),p_{h}-\widetilde{p}_{h})_{0,\Omega}\vspace{2mm}\\
&\leq&||A^{-1/2}(p-p_{h})||_{L^{2}(\Omega)}
||A^{-1/2}(p_{h}-\widetilde{p}_{h})||_{L^{2}(\Omega)}\vspace{2mm}\\
&\leq&C_{0}^{1/2}||A^{-1/2}(p-p_{h})||_{L^{2}(\Omega)}{\rm
osc}(f_{h},\mathcal{T}_{H}),
\end{array}
\end{equation}
which implies the first result (\ref{Lemma 4.5}).

Furthermore, notice that
\begin{eqnarray}\label{equality}
||A^{-1/2}(p-p_{h})||_{L^{2}(\Omega)}^{2}&=&
||A^{-1/2}(p-p_{H})||_{L^{2}(\Omega)}^{2}-
||A^{-1/2}(p_{h}-p_{H})||_{L^{2}(\Omega)}^{2}\nonumber\\
& & -2(A^{-1}(p-p_{h}),p_{h}-p_{H})_{0,\Omega},
\end{eqnarray}
then for any $\delta_{1}>0$, from 
(\ref{Lemma 4.5.4}) and Young's inequality we have
\begin{equation*}
\begin{array}{lll}
||A^{-1/2}(p-p_{h})||_{L^{2}(\Omega)}^{2}
&\leq&||A^{-1/2}(p-p_{H})||_{L^{2}(\Omega)}^{2}-
||A^{-1/2}(p_{h}-p_{H})||_{L^{2}(\Omega)}^{2}\vspace{2mm}\\
&\ & +\delta_{1}||A^{-1/2}(p-p_{h})||_{L^{2}(\Omega)}^{2}+
\frac{C_{0}}{\delta_{1}}{\rm osc}^{2}(f_{h},\mathcal{T}_{H}),
\end{array}
\end{equation*}
which implies the estimate (\ref{Lemma 4.5.1}). 

In
particular, if ${\rm osc}(f_{h},\mathcal{T}_{H})=0$, then from (\ref{Lemma 4.5.4}) it follows
$(A^{-1}(p-p_{h}),p_{h}-p_{H})_{0,\Omega}=0$.  This, together with (\ref{equality}), yields 
the relation (\ref{Lemma 4.5.2}).
\end{proof}

Although the oscillation of $f_{h}$ over the triangulation
$\mathcal{T}_{H}$ appears in the estimate of quasi-orthogonality, it
is dominated by ${\rm osc}(f,\mathcal{T}_{H})$. We refer to
\cite{Chen;Holst;Xu} for the proof of the following observation.

\begin{lemma}\label{lem 4.6}  Let $f_{h}$ denote the $L^{2}-$projection of
$f$ over $L_{h}$, then it holds
\begin{equation*}
{\rm osc}(f_{h},\mathcal{T}_{H})\leq{\rm osc}(f,\mathcal{T}_{H}).
\end{equation*}
\end{lemma}
\subsection{Estimator and oscillation reduction}
In this subsection, we aim at reduction of estimator and
oscillation. To this end, we relate the error indicators and
oscillation of two nested triangulations to each other. The link
involves weighted maximum-norms of the inverse matrix, $A^{-1}$, of
coefficient matrix $A$ and its oscillation. 

For a nonnegative integer $m=n-l$, any given triangulation
$\mathcal{T}_{H}$, and $v\in L^{\infty}(\Omega)$, we denote by
$\Pi_{m}^{\infty}v$  the best $L^{\infty}(\Omega)-$approximation of $v$
in the space of piecewise polynomials of degree $\leq m$, and denote
by $\omega_{T}$ the union of elements in $\mathcal{T}_{H}$ sharing a
edge with $T$. We further set $$\Pi_{-1}^{\infty}v:=0,\ \ 
P_{m}^{\infty}v:=(id-\Pi_{m}^{\infty})v,$$
\begin{equation*}
\eta_{\mathcal{T}_{H}}^{2}(A^{-1},T) :=H_{T}^{2}(||{\rm Curl}\
A^{-1}||_{L^{\infty}(T)}^{2}+H_{T}^{-2}||A^{-1}||_{L^{\infty}(\omega_{T})}^{2})\
\ {\rm for\ all\ }\ T\in\mathcal{T}_{H},
\end{equation*}
\begin{equation*}
{\rm osc}_{\mathcal{T}_{H}}^{2}(A^{-1},T)
:=H_{T}^{2}(||P_{m-1}^{\infty}{\rm Curl}\
A^{-1}||_{L^{\infty}(T)}^{2}+
H_{T}^{-2}||P_{m}^{\infty}A^{-1}||_{L^{\infty}(\omega_{T})}^{2}).
\end{equation*}
Noticing that $P_{m}^{\infty}$ is defined elementwise, for any
subset $\mathcal{T}_{H}'\subset\mathcal{T}_{H}$ we finally set
\begin{equation*}
\eta_{\mathcal{T}_{H}}(A^{-1},\mathcal{T}_{H}')
:=\max_{T\in\mathcal{T}_{H}'}\eta_{\mathcal{T}_{H}}(A^{-1},T),\ {\rm
osc}_{\mathcal{T}_{H}}(A^{-1},\mathcal{T}_{H}')
:=\max_{T\in\mathcal{T}_{H}'}{\rm osc}_{\mathcal{T}_{H}}(A^{-1},T).
\end{equation*}

\begin{remark}\label{rem 4.7}({\rm Monotonicity}) The use of best
approximation in $L^{\infty}$ in the definition of
$\eta_{\mathcal{T}_{H}}(A^{-1},\mathcal{T}_{H})$ and ${\rm
osc}_{\mathcal{T}_{H}}(A^{-1},\mathcal{T}_{H})$ implies the
following monotonicity: for any refinement $\mathcal{T}_{h}$ of
$\mathcal{T}_{H}$, it holds
\begin{equation*}
\eta_{\mathcal{T}_{h}}(A^{-1},\mathcal{T}_{h})\leq
\eta_{\mathcal{T}_{H}}(A^{-1},\mathcal{T}_{H})\ \ {\rm and}\ \ {\rm
osc}_{\mathcal{T}_{h}}(A^{-1},\mathcal{T}_{h})\leq{\rm
osc}_{\mathcal{T}_{H}}(A^{-1},\mathcal{T}_{H}).
\end{equation*}
\end{remark}

To avoid any smoothness assumptions on the coefficient matrix of
PDEs, we need to quote a result about implicit interpolation, whose
proof can be found in \cite{Cascon;Kreuzer;Nochetto;Siebert}.

\begin{lemma}\label{lem 4.8}({\rm Implicit interpolation})  Let
$\bar{m}$ and $\bar{n}$ be two nonnegative integer, and $\omega$ be
either one or two dimension simplex. For a positive integer $\iota$
we denote by $\Pi_{\bar{m}}^{2}
:L^{2}(\omega,\mathbb{R}^{\iota})\rightarrow
P_{\bar{m}}(\omega,\mathbb{R}^{\iota})$ the operator of best
$L^{2}-$approximation in $\omega$, and $P_{\bar{m}}^{2}
:=id-\Pi_{\bar{m}}^{2}$. Then for all $v\in
L^{\infty}(\omega,\mathbb{R}^{\iota})$, $V\in
P_{\bar{n}}(\omega,\mathbb{R}^{\iota})$ and $\bar{m}\geq \bar{n}$,
it holds
\begin{equation}\label{Lemma 4.8}
||P_{\bar{m}}^{2}(vV)||_{L^{2}(\omega)}\leq
||P_{\bar{m}-\bar{n}}^{\infty}v||_{L^{\infty}(\omega)}||V||_{L^{2}(\omega)}.
\end{equation}
\end{lemma}

\begin{lemma}\label{lem 4.9} Let $\mathcal{T}_{H}$ be a triangulation.
For all $T\in\mathcal{T}_{H}$ and   any pair of discrete functions
$\sigma_{H}, \tau_{H}\in M_{H}$, there exists a constant
$\bar{\Lambda}_{1}>0$ depending only on the shape regularity of
$\mathcal{T}_{0}$, the polynomial degree $l+1$, and the eigenvalues
of $A^{-1}$, such that
\begin{equation}\label{Lemma 4.9.1}
\eta_{\mathcal{T}_{H}}(\sigma_{H},T)\leq\eta_{\mathcal{T}_{H}}(\tau_{H},T)+
\bar{\Lambda}_{1}\eta_{\mathcal{T}_{H}}(A^{-1},T)||A^{-1/2}(\sigma_{H}-
\tau_{H})||_{L^{2}(\omega_{T})},
\end{equation}
\begin{equation}\label{Lemma 4.9.2}
{\rm osc}_{\mathcal{T}_{H}}(\sigma_{H},T)\leq{\rm
osc}_{\mathcal{T}_{H}}(\tau_{H},T)+\bar{\Lambda}_{1}{\rm
osc}_{\mathcal{T}_{H}}(A^{-1},T)||A^{-1/2}(\sigma_{H}-
\tau_{H})||_{L^{2}(\omega_{T})}.
\end{equation}
\end{lemma}
\begin{proof} We  only prove the second estimate (\ref{Lemma 4.9.2}),
since the first one (\ref{Lemma 4.9.1}) is somewhat simpler and can
be derived similarly. We denote by $L^{2}(\Gamma_{H})$ the square
integrable function spaces on $\Gamma_{H}
:=\bigcup\varepsilon_{H}$. The jump of the tangential component
defines a linear mapping $J : M_{H}\rightarrow L^{2}_{\Gamma_{H}}$
by $J(q_{H})=J(A^{-1}q_{H}\cdot\tau)$ for all $q_{H}\in M_{H}$ from
$M_{H}$ into $L^{2}_{\Gamma_{H}}$ . Recalling 
$P_{n}^{2}=id-\Pi_{n}^{2}$ with $\Pi_{n}^{2}$ being the $L^{2}-$projection,   denoting $q_{H}:=\sigma_{H}-\tau_{H}$ and using the triangle
inequality, we have
\begin{equation}\label{Lemma 4.9.3}
\begin{array}{lll}
{\rm osc}_{\mathcal{T}_{H}}(\sigma_{H},T)&\leq&{\rm
osc}_{\mathcal{T}_{H}}(\tau_{H},T)+H_{T}||P_{n}^{2}{\rm
curl}(A^{-1}q_{H})||_{L^{2}(T)}\vspace{2mm}\\
&\ &
+H_{T}^{1/2}||P_{n+1}^{2}J(A^{-1}q_{H}\cdot\tau)||_{L^{2}(\partial
T)}.
\end{array}
\end{equation}
  We split the ${\rm curl}$ term as
\begin{equation*}
{\rm curl}(A^{-1}q_{H})={\rm Curl}\ A^{-1}\cdot
q_{H}+A^{-1}:\widetilde{{\rm curl}}\ q_{H},
\end{equation*}
where ${\rm Curl}\ A^{-1}$ is a vector whose every component is the
${\rm curl}$ of the corresponding column vector of $A^{-1}$, and
$\widetilde{{\rm curl}}\ q_{H}$ is a matrix whose column vector is
the ${\rm Curl}$ of the corresponding vector of $q_{H}$. Invoking
Lemma \ref{lem 4.8} with $\omega=T$ and noticing that polynomial degree of
$q_{H}$ is $l+1$, we infer for the first term that
\begin{equation}\label{Lemma 4.9.4}
||P_{n}^{2}({\rm curl}\ A^{-1}\cdot
q_{H})||_{L^{2}(T)}\lesssim||P_{n-l-1}^{\infty}{\rm curl}\
A^{-1}||_{L^{\infty}(T)}||A^{-1/2}q_{H}||_{L^{2}(T)}.
\end{equation}

Since $\widetilde{{\rm curl}}\ q_{H}$ is a polynomial of degree
$\leq l$,  applying (\ref{Lemma 4.8}) again in conjunction with an
inverse inequality, we obtain for the second term that
\begin{equation}\label{Lemma 4.9.5}
\begin{array}{lll}
||P_{n}^{2}(A^{-1}:\widetilde{{\rm curl}}\
q_{H})||_{L^{2}(T)}&\leq&||P_{n-l}^{\infty}A^{-1}||_{L^{\infty}(T)}||\widetilde{{\rm
curl}}\ q_{H}||_{L^{2}(T)}\vspace{2mm}\\
&\leq&||P_{n-l}^{\infty}A^{-1}||_{L^{\infty}(T)}|q_{h}|_{H^{1}(T)}\vspace{2mm}\\
&\lesssim&
H_{T}^{-1}||P_{n-l}^{\infty}A^{-1}||_{L^{\infty}(T)}||A^{-1/2}q_{H}||_{L^{2}(T)}.
\end{array}
\end{equation}

We now deal with the jump residual. Let $T'\in\mathcal{T}_{H}$ share
an interior edge $E$ with $T$. We write
$J(A^{-1}q_{H}\cdot\tau)=((A^{-1}q_{H})|_{T}-(A^{-1}q_{H})|_{T'})\cdot\tau$
and use the linearity of $\Pi_{n+1}^{2}$, Lemma \ref{lem 4.8} with $\omega=E$,
and the inverse inequality $||q_{H}||_{L^{2}(E)}\lesssim
H_{T}^{-1/2}||q_{H}||_{L^{2}(T)}$ to deduce that
\begin{equation}\label{Lemma 4.9.6}
\begin{array}{lll}
&\ &||P_{n+1}^{2}((A^{-1}q_{H})|_{T}\cdot\tau)||_{L^{2}(E)}\vspace{2mm}\\
&\ &=||(P_{n+1}^{2}(A^{-1}q_{H}|_{T}))\cdot\tau||_{L^{2}(E)}\leq
||P_{n+1}^{2}(A^{-1}q_{H}|_{T})||_{L^{2}(E)}\vspace{2mm}\\
&\
&\leq||P_{n-l}^{\infty}A^{-1}|_{T}||_{L^{\infty}(E)}||q_{H}||_{L^{2}(E)}\lesssim
H_{T}^{-1/2}||P_{n-l}^{\infty}A^{-1}||_{L^{\infty}(T)}||q_{H}||_{L^{2}(T)}.
\end{array}
\end{equation}

Since $\mathcal{T}_{H}$ is shape-regular, we can replace $H_{T}'$ by
$H_{T}$, a similar argument leads to
\begin{equation}\label{Lemma 4.9.7}
||P_{n+1}^{2}((A^{-1}q_{H})|_{T'}\cdot\tau)||_{L^{2}(E)}\lesssim
H_{T}^{-1/2}||P_{n-l}^{\infty}A^{-1}||_{L^{\infty}(T')}||q_{H}||_{L^{2}(T')}.
\end{equation}
A combination of (\ref{Lemma 4.9.6}) and (\ref{Lemma 4.9.7}) yields
\begin{equation}\label{Lemma 4.9.8}
\begin{array}{lll}
&&||P_{n+1}^{2}J(A^{-1}q_{H}\cdot\tau)||_{L^{2}(E)}\vspace{2mm}\\
&\ & =||P_{n+1}^{2}(((A^{-1}q_{H})|_{T}-(A^{-1}q_{H})|_{T'})\cdot\tau)||_{L^{2}(E)}\vspace{2mm}\\
&\ & \leq||P_{n+1}^{2}((A^{-1}q_{H})|_{T}\cdot\tau)||_{L^{2}(E)}+
||P_{n+1}^{2}((A^{-1}q_{H})|_{T'}\cdot\tau)||_{L^{2}(E)}\vspace{2mm}\\
&\ & \lesssim
H_{T}^{-1/2}||P_{n-l}^{\infty}A^{-1}||_{L^{\infty}(\omega_{E})}
||A^{-1/2}q_{H}||_{L^{2}(\omega_{E})}.
\end{array}
\end{equation}
By summing over all edges of element $T$, from the above inequality
(\ref{Lemma 4.9.8}), we get
\begin{equation}\label{Lemma 4.9.9}
||P_{n+1}^{2}J(A^{-1}q_{H}\cdot\tau)||_{L^{2}(\partial T)}\lesssim
H_{T}^{-1/2}||P_{n-l}^{\infty}A^{-1}||_{L^{\infty}(\omega_{T})}
||A^{-1/2}q_{H}||_{L^{2}(\omega_{T})}.
\end{equation}
Finally, the desired result (\ref{Lemma 4.9.2}) follows from   (\ref{Lemma 4.9.3})-(\ref{Lemma 4.9.5}) and
(\ref{Lemma 4.9.9}).
\end{proof}

The following two corollaries are global forms of the above lemma.

\begin{corollary}\label{cor 4.10}({\rm Estimator reduction})   For a
triangulation $\mathcal{T}_{H}$  with
$\mathcal{M}_{H}\subset\mathcal{T}_{H}$, let $\mathcal{T}_{h}$ be a
refinement of $\mathcal{T}_{H}$ obtained by $\mathcal{T}_{h}
:=REFINE(\mathcal{T}_{H},\mathcal{M}_{H})$.  Denote $\Lambda_{1}
:=3\bar{\Lambda}_{1}^{2}$ with $\bar{\Lambda}_{1}$ given in Lemma \ref{lem 4.9}, and $\lambda :=1-2^{-b/2}>0$ with $b$ given in
Section 2.7. Then it holds
\begin{equation}\label{Lemma 4.10}
\begin{array}{lll}
\eta_{\mathcal{T}_{h}}^{2}(\sigma_{h},\mathcal{T}_{h})&\leq&(1+\delta_{3})
\{\eta_{\mathcal{T}_{H}}^{2}(\sigma_{H},\mathcal{T}_{H})-
\lambda\eta_{\mathcal{T}_{H}}^{2}(\sigma_{H},\mathcal{M}_{H})\}\vspace{2mm}\\
&\ &
+(1+\delta_{3}^{-1})\Lambda_{1}\eta_{\mathcal{T}_{0}}^{2}(A^{-1},\mathcal{T}_{0})
||A^{-1/2}(\sigma_{H}-\sigma_{h})||_{L^{2}(\Omega)}^{2}
\end{array}
\end{equation}
for all $\sigma_{H}\in M_H, \sigma_{h}\in M_{h}$ and any
$\delta_{3}>0$.
\end{corollary}
\begin{proof} For $T\in\mathcal{T}_{h}$,  applying the first estimate (\ref{Lemma 4.9.1}) of Lemma \ref{lem 4.9}
with $\sigma_{H},\sigma_{h}\in M_{h}$ and using Young's inequality
with parameter $\delta_{3}>0$, we derive
\begin{equation}\label{Lemma 4.10.1}
\begin{array}{lll}
\eta_{\mathcal{T}_{h}}^{2}(\sigma_{h},T)&\leq&(1+\delta_{3})
\eta_{\mathcal{T}_{h}}^{2}(\sigma_{H},T)\vspace{2mm}\\
&\ & +(1+\delta_{3}^{-1})\bar{\Lambda}_{1}^{2}
\eta_{\mathcal{T}_{h}}^{2}(A^{-1},T)
||A^{-1/2}(\sigma_{H}-\sigma_{h})||_{L^{2}(\omega_{T})}^{2}.
\end{array}
\end{equation}

By summing over all elements $T\in\mathcal{T}_{h}$ and using the finite
overlap of patches $\omega_{T}$,  the above inequality
(\ref{Lemma 4.10.1}) indicates 
\begin{equation}\label{Lemma 4.10.2}
\begin{array}{lll}
\eta_{\mathcal{T}_{h}}^{2}(\sigma_{h},\mathcal{T}_{h})&\leq&(1+\delta_{3})
\eta_{\mathcal{T}_{h}}^{2}(\sigma_{H},\mathcal{T}_{h})\vspace{2mm}\\
&\ & +(1+\delta_{3}^{-1})3\bar{\Lambda}_{1}^{2}
\eta_{\mathcal{T}_{h}}^{2}(A^{-1},\mathcal{T}_{h})
||A^{-1/2}(\sigma_{H}-\sigma_{h})||_{L^{2}(\Omega)}^{2}.
\end{array}
\end{equation}

For a marked element $T\in\mathcal{M}_{H}$, we set
$\mathcal{T}_{h,T} :=\{T'\in\mathcal{T}_{h}|T'\subset T\}$. Since
$\sigma_{H}\in M_{H}$ and $A^{-1}$ jumps only across edges of
$\mathcal{T}_{0}$, we have $J(A^{-1}\sigma_{H}\cdot\tau)=0$ on
edges of $\mathcal{T}_{h,T}$ in the interior of $T$. Notice that
$||f-f_{h}||_{L^{2}(T')}\leq||f-f_{H}||_{L^{2}(T')}$, we then obtain
\begin{equation}\label{Lemma 4.10.3}
\sum\limits_{T'\in\mathcal{T}_{h,T}}\eta_{\mathcal{T}_{h}}^{2}(\sigma_{H},T')\leq
2^{-b/2}\eta_{\mathcal{T}_{H}}^{2}(\sigma_{H},T),
\end{equation}
since refinement by bisection implies 
$$h_{T'}=|T'|^{1/2}\leq(2^{-b}|T|)^{1/2}\leq2^{-b/2}H_{T}\ \ \text{ for all }
T'\in\mathcal{T}_{h,T}.$$ 
On the other hand, for an element
$T\in\mathcal{T}_{H}\setminus\mathcal{M}_{H}$,   Remark 
ef{rem 2.1}
yields
$$\eta_{\mathcal{T}_{h}}(\sigma_{H},T)\leq\eta_{\mathcal{T}_{H}}(\sigma_{H},T).$$
Hence, from (\ref{Lemma 4.10.3}) and the above inequality,  by summing over all
$T\in\mathcal{T}_{h}$ we arrive at
\begin{equation}\label{Lemma 4.10.4}
\begin{array}{lll}
\eta_{\mathcal{T}_{h}}^{2}(\sigma_{H},\mathcal{T}_{h})&\leq&2^{-b/2}
\eta_{\mathcal{T}_{H}}^{2}(\sigma_{H},\mathcal{M}_{H})+
\eta_{\mathcal{T}_{H}}^{2}(\sigma_{H},\mathcal{T}_{H}\setminus\mathcal{M}_{H})\vspace{2mm}\\
&=&\eta_{\mathcal{T}_{H}}^{2}(\sigma_{H},\mathcal{T}_{H})-
\lambda\eta_{\mathcal{T}_{H}}^{2}(\sigma_{H},\mathcal{M}_{H}).
\end{array}
\end{equation}

From (\ref{Lemma 4.10.2}), (\ref{Lemma 4.10.4}), and the
monotonicity $\eta_{\mathcal{T}_{h}}(A^{-1},\mathcal{T}_{h})\leq
\eta_{\mathcal{T}_{0}}(A^{-1},\mathcal{T}_{0})$ stated in Remark
\ref{rem 4.7}, we get the desired result (\ref{Lemma 4.10}).
\end{proof}

\begin{corollary}\label{cor 4.11}({\rm Perturbation of oscillation}) 
Let $\mathcal{T}_{h}$ be a refinement of $\mathcal{T}_{H}$, and let
$\Lambda_{1}$ be the same as in Corollary \ref{cor 4.10}. Then for
all $\sigma_{H}\in M_h,\ \sigma_{h}\in M_{h}$,  it holds
\begin{equation*}
\begin{array}{lll}
{\rm
osc}_{\mathcal{T}_{H}}^{2}(\sigma_{H},\mathcal{T}_{H}\cap\mathcal{T}_{h})&\leq&2
{\rm
osc}_{\mathcal{T}_{h}}^{2}(\sigma_{h},\mathcal{T}_{H}\cap\mathcal{T}_{h})\vspace{2mm}\\
&\ & +2\Lambda_{1}{\rm
osc}_{\mathcal{T}_{0}}^{2}(A^{-1},\mathcal{T}_{0})
||A^{-1/2}(\sigma_{h}-\sigma_{H})||_{L^{2}(\Omega)}^{2}.
\end{array}
\end{equation*}
\end{corollary}
\begin{proof}
Remark 
\ref{rem 2.1} yields ${\rm osc}_{\mathcal{T}_{H}}(\sigma_{H},T)={\rm
osc}_{\mathcal{T}_{h}}(\sigma_{H},T)$ for all
$T\in\mathcal{T}_{H}\cap\mathcal{T}_{h}$. Hence, by the
estimate (\ref{Lemma 4.9.2})  and Young's inequality, we get
\begin{equation}\label{Lemma 4.11}
{\rm osc}_{\mathcal{T}_{H}}^{2}(\sigma_{H},T)\leq2 {\rm
osc}_{\mathcal{T}_{h}}^{2}(\sigma_{h},T)+2\bar{\Lambda}_{1}^{2}{\rm
osc}_{\mathcal{T}_{h}}^{2}(A^{-1},\mathcal{T}_{h})
||A^{-1/2}(\sigma_{h}-\sigma_{H})||_{L^{2}(\omega_{T})}^{2}.
\end{equation}

By summing over $T\in\mathcal{T}_{H}\cap\mathcal{T}_{h}$  and using
the monotonicity property ${\rm
osc}_{\mathcal{T}_{h}}(A^{-1},\\\mathcal{T}_{h})\leq{\rm
osc}_{\mathcal{T}_{0}}(A^{-1},\mathcal{T}_{0})$ stated in Remark \ref{rem
4.7}, the  inequality (\ref{Lemma 4.11}) indicates the desired
assertion.
\end{proof}

\section{Convergence for the AMFEM}
We shall prove in this section that the so-called quasi-error,  i.e., the sum of the
stress variable error plus the scaled estimator, uniformly reduces with a fixed rate on
two successive meshes, up to an oscillation term of $f$. This means the AMFEM is a contraction
with respect to the quasi-error. To this
end, subsequently we replace the subscripts $H, h$ respectively with  iteration
counters   $k, k+1$, 
and denote by $$\eta_{k}
:=\eta_{\mathcal{T}_{k}}(p_{k},\mathcal{T}_{k})$$ the scaled
estimator over the whole mesh $\mathcal{T}_{k}$.

\begin{theorem}\label{thm 5.1}({\rm Contraction property})  Given
$\theta\in(0,1]$, let
$\{\mathcal{T}_{k};({M}_{k},L_{k});(p_{k},\\u_{k})\}_{k\geq0}$ be
the sequence of meshes,a pair of finite element spaces, and discrete
solutions produced by the AMFEM. Then there exits constants
$\gamma>0$, $0<\alpha<1$, and $C>0$ depending solely on the
shape-regularity of $\mathcal{T}_{0}$, $b$,
$\eta_{\mathcal{T}_{0}(A^{-1},\mathcal{T}_{0})}$, and the marking
parameter $\theta$, such that
\begin{equation}\label{Theorem 5.1}
\mathcal{E}_{k+1}^{2}+\gamma\eta_{k+1}^{2}\leq\alpha^{2}(\mathcal{E}_{k}^{2}+\gamma\eta_{k}^{2})+
C{\rm osc}^{2}(f,\mathcal{T}_{k}).
\end{equation}
\end{theorem}
\begin{proof} For convenience, we use the notations $\epsilon_{k}
:=p-p_{k}$, $E_{k} :=p_{k}-p_{k+1}$, $\eta_{k}(\mathcal{M}_{k})
:=\eta_{\mathcal{T}_{k}}(p_{k},\mathcal{M}_{k})$, $\eta_{0}(A^{-1})
:=\eta_{\mathcal{T}_{0}}(A^{-1},\mathcal{T}_{0})$. 

For any
$\delta_{2}>0$, by Young's inequality and the mesh-size functions
$h_{k+1}\leq h_{k}$, we have
\begin{equation*}
\begin{array}{lll}
||h_{k+1}{\rm div}\
\epsilon_{k+1}||_{L^{2}(\Omega)}^{2}&=&||h_{k+1}{\rm div}\
\epsilon_{k}||_{L^{2}(\Omega)}^{2}-||h_{k+1}{\rm
div}\ E_{k}||_{L^{2}(\Omega)}^{2}\vspace{2mm}\\
&\ & -2(h_{k+1}{\rm div}\ \epsilon_{k+1},h_{k+1}{\rm
div}\ E_{k})\vspace{2mm}\\
&\leq&||h_{k}{\rm div}\
\epsilon_{k}||_{L^{2}(\Omega)}^{2}-||h_{k+1}{\rm
div}\ E_{k}||_{L^{2}(\Omega)}^{2}\vspace{2mm}\\
&\ & +\delta_{2}||h_{k+1}{\rm div}\
\epsilon_{k+1}||_{L^{2}(\Omega)}^{2}+\delta_{2}^{-1}{\rm
osc}^{2}(f_{k+1},\mathcal{T}_{k}),
\end{array}
\end{equation*}
which implies the  inequality
\begin{equation}\label{Theorem 5.1.1}
\begin{array}{lll}
(1-\delta_{2})||h_{k+1}{\rm div}\
\epsilon_{k+1}||_{L^{2}(\Omega)}^{2}&\leq&||h_{k}{\rm div}\
\epsilon_{k}||_{L^{2}(\Omega)}^{2}-||h_{k+1}{\rm div}\
E_{k}||_{L^{2}(\Omega)}^{2}\vspace{2mm}\\
&\ & +\delta_{2}^{-1}{\rm osc}^{2}(f_{k+1},\mathcal{T}_{k}),
\end{array}
\end{equation}
where $f_{k+1} :=\Pi_{L_{k+1}}f$ is the $L^{2}-$projection of $f$
over $L_{k+1}$.

We combine the quasi-orthogonality (Theorem \ref{thm 4.5}) and (\ref{Theorem
5.1.1}), and take $\delta_{2}=\delta_{1}$ to obtain
\begin{equation}\label{Theorem 5.1.2}
(1-\delta_{1})\mathcal{E}_{k+1}^{2}\leq
\mathcal{E}_{k}^{2}-||A^{-1/2}E_{k}||_{L^{2}(\Omega)}^{2}+\frac{C_{0}+1}{\delta_{1}}{\rm
osc}^{2}(f_{k+1},\mathcal{T}_{k}).
\end{equation}
Applying the estimator reduction (Corollary \ref{cor 4.10}) to
(\ref{Theorem 5.1.2}), we get for any $\bar{\gamma}\geq 0$,
\begin{equation}\label{Theorem 5.1.3}
\begin{array}{lll}
(1-\delta_{1})\mathcal{E}_{k+1}^{2}+\bar{\gamma}\eta_{k+1}^{2}&\leq&
\mathcal{E}_{k}^{2}-||A^{-1/2}E_{k}||_{L^{2}(\Omega)}^{2}\vspace{2mm}\\
&\ & +\bar{\gamma}(1+\delta_{3})\{
\eta_{k}^{2}-\lambda\eta_{k}^{2}(\mathcal{M}_{k})\}\vspace{2mm}\\
&\ & +\bar{\gamma}(1+
\delta_{3}^{-1})\Lambda_{1}\eta_{0}^{2}(A^{-1})||A^{-1/2}E_{k}||_{L^{2}(\Omega)}^{2}\vspace{2mm}\\
&\ & +\frac{C_{0}+1}{\delta_{1}}{\rm
osc}^{2}(f_{k+1},\mathcal{T}_{k}).
\end{array}
\end{equation}
In what follows we choose 
\begin{equation}\label{Theorem 5.1.4}
\bar{\gamma}:=1/\left((1+\delta_{3}^{-1})\Lambda_{1}\eta_{0}^{2}(A^{-1})\right)
\end{equation}
so as to obtain
$$
(1-\delta_{1})\mathcal{E}_{k+1}^{2}+\bar{\gamma}\eta_{k+1}^{2}\leq
(1-\delta_{1})\mathcal{E}_{k}^{2}+\bar{\gamma}(1+ \delta_{3})\{
\eta_{k}^{2}-\lambda\eta_{k}^{2}(\mathcal{M}_{k})\}
+\frac{C_{0}+1}{\delta_{1}}{\rm osc}^{2}(f_{k+1},\mathcal{T}_{k}).
$$
By using the reliable estimation (\ref{Section 2.4}) of the stress
variable error, and invoking D\"{o}rfler marking property
(\ref{Dorfler property}), the above inequality yields for any constant $\alpha$,
{\small
\begin{eqnarray}
(1-\delta_{1})\mathcal{E}_{k+1}^{2}+\bar{\gamma}\eta_{k+1}^{2}&\leq&
\alpha^{2}(1-\delta_{1})\mathcal{E}_{k}^{2}+(1-\alpha^{2}(1-\delta_{1}))\mathcal{E}_{k}^{2}\nonumber\\
&&\ +\bar{\gamma}(1+ \delta_{3})\{
\eta_{k}^{2}-\lambda\eta_{k}^{2}(\mathcal{M}_{k})\}
+\frac{C_{0}+1}{\delta_{1}}{\rm osc}^{2}(f_{k+1},\mathcal{T}_{k})\nonumber\\
 &\leq&
\alpha^{2}(1-\delta_{1})\mathcal{E}_{k}^{2}+(1-\alpha^{2}(1-\delta_{1}))C_{1}\eta_{k}^{2}\nonumber\\
& &\ +\bar{\gamma}(1+ \delta_{3})(1-\lambda\theta^{2})\eta_{k}^{2}
+\frac{C_{0}+1}{\delta_{1}}{\rm
osc}^{2}(f_{k+1},\mathcal{T}_{k})\nonumber\\
&\leq&\alpha^{2}\{(1-\delta_{1})\mathcal{E}_{k}^{2}+\frac{(1-\alpha^{2}(1-\delta_{1}))C_{1}+\bar{\gamma}(1+
\delta_{3})(1-\lambda\theta^{2})}{\alpha^{2}}\eta_{k}^{2}\}\nonumber\\
& & \ +\frac{C_{0}+1}{\delta_{1}}{\rm
osc}^{2}(f_{k+1},\mathcal{T}_{k}).\label{Theorem 5.1.5}
\end{eqnarray}
}
We   choose $\alpha$ such that $(1-\alpha^{2}(1-\delta_{1}))C_{1}+\bar{\gamma}(1+
\delta_{3})(1-\lambda\theta^{2})=\alpha^{2}\bar{\gamma}$, which
indicates
\begin{equation*}
\alpha^{2}=\frac{(1-\delta_{1})C_{1}+\bar{\gamma}(\delta_{1}C_{1}/\bar{\gamma}+
(1+\delta_{3})(1-\lambda\theta^{2}))}{(1-\delta_{1})C_{1}+\bar{\gamma}}.
\end{equation*}
We now      choose $\delta_{3}$  and  $\delta_{1}$ such that
$$\delta_{3}\leq\lambda\theta^{2}/2(1-\lambda\theta^{2})\ \text{ and } \ \delta_{1}<\min\{1,\frac{(1-\lambda\theta^{2})\delta_{3}\bar{\gamma}}{C_{1}}\}.$$
Then it follows
\begin{equation*}
\begin{array}{lll}
\delta_{1}C_{1}/\bar{\gamma}+
(1+\delta_{3})(1-\lambda\theta^{2})&<&(1-\lambda\theta^{2})\delta_{3}+
(1+\delta_{3})(1-\lambda\theta^{2})\vspace{2mm}\\
&\leq&(1-\lambda\theta^{2})(1+\frac{\lambda\theta^{2}}{1-\lambda\theta^{2}})=1,
\end{array}
\end{equation*}
which leads to $\alpha^{2}<1$.
Finally we set
$\gamma=\bar{\gamma}/(1-\delta_{1})$. Then the  desired result
(\ref{Theorem 5.1})  follows from (\ref{Theorem 5.1.5}).
\end{proof}

We note that the oscillation ${\rm osc}(f,\mathcal{T}_{k})$ of the
right-hand side term $f$ measures intrinsic information missing in
the average process associated with finite elements, but fails to
detect fine structures of $f$. When the oscillation
term ${\rm osc}(f,\mathcal{T}_{k})$ is marked, it is easy to show
the following convergence result.
\begin{corollary}\label{cor 5.2}({\rm Convergence result}) Under the
assumptions of Theorem \ref{thm 5.1}, there exit constants $\rho\in(0,1)$,
$\gamma>0$, and $C>0$ depending solely on the shape-regularity of
$\mathcal{T}_{0}$, $b$,
$\eta_{\mathcal{T}_{0}(A^{-1},\mathcal{T}_{0})}$, and the marking
parameter $\theta$, such that
\begin{equation*}
\mathcal{E}_{k}^{2}+\gamma\eta_{k}^{2}\leq C\rho^{2k}.
\end{equation*}
\end{corollary}

\section{Quasi-optimal convergence rate for the AMFEM}
\subsection{Auxiliary results}
In this subsection, we aim at the discrete upper bound, which is one
key for the proof for the quasi-optimal convergence rate.
Simultaneously, we shall prove and quote some preliminary results.

\begin{theorem}\label{thm 6.1}({\rm Discrete upper bound})  Let
$\mathcal{T}_{h}$ and $\mathcal{T}_{H}$ be two nested conforming
triangulations, $(p_{h},u_{h})\in M_{h}\times L_{h}$ and
$(p_{H},u_{H})\in M_{H}\times L_{H}$ be the discrete solutions with
respect to the meshes $\mathcal{T}_{h}$ and $\mathcal{T}_{H}$,
respectively, and $\mathcal{F}_{H} :=\{T\in\mathcal{T}_{H}: T\ {\rm
is\ not\ included\ in}\ \mathcal{T}_{h}\}$. Then there exist
constants $C_{1}$ and $C_{0}$ depending only on the shape regularity
of $\mathcal{T}_{H}$ such that
\begin{equation}\label{Theorem 6.1.1}
||A^{-1/2}(p_{h}-p_{H})||_{L^{2}(\Omega)}^{2}\leq
C_{1}\eta_{\mathcal{T}_{H}}^{2}(p_{H},\mathcal{F}_{H})+C_{0}{\rm
osc}^{2}(f_{h},\mathcal{T}_{H}),
\end{equation}
and
\begin{equation}\label{Theorem 6.1.2}
\#\mathcal{F}_{H}\leq\#\mathcal{T}_{h}-\#\mathcal{T}_{H}.
\end{equation}
\end{theorem}
\begin{proof} The second inequality, i.e., (\ref{Theorem 6.1.2}), follows from
 the definition of
$\mathcal{F}_{H}$. To prove the first one, we introduce the solution
$(\widetilde{p}_{h},\widetilde{u}_{h})\in M_{h}\times L_{h}$ to the
problem (\ref{quasi-optimality 2.4}). From (\ref{quasi-optimality
2.3}) and (\ref{quasi-optimality 2.4}), we obtain ${\rm
div}(\widetilde{p}_{h}-p_{H})=0$, which implies
\begin{equation}\label{Theorem 6.1.3}
\int_{\partial\Omega}(\widetilde{p}_{h}-p_{H})\cdot\nu ds=0,
\end{equation}
where $\nu$ is the outward unit normal vector along
$\partial\Omega$.
Thus $\widetilde{p}_{h}-p_{H}$ satisfies the conditions of Theorem
3.1 in \cite{Girault} on the polygonal domain $\Omega$, namely it is
divergence-free and fulfills (\ref{Theorem 6.1.3}). As a result,
there exists $\psi_{h}\in H^1(\Omega)$ with $\widetilde{p}_{h}-p_{H}={\rm Curl}\
\psi_{h}$.  Since $\widetilde{p}_{h}-p_{H}\in M_{h}$, this leads to
$$\psi_{h}\in S_{h}^{l+2} :=\{\psi_{h}\in C(\overline{\Omega}): {\rm Curl}\
\psi_{h}\in M_{h},\psi_{h}|_{T}\in P_{l+2}(T) \text{ for all }
T\in\mathcal{T}_{h}\}$$ 
(the definition of $S_{H}^{l+2}$ is
analogous).  From (\ref{quasi-optimality 2.4}) with
$q_{h}=\widetilde{p}_{h}-p_{H}$, we get
\begin{equation}\label{Theorem 6.1.4}
\begin{array}{lll}
||A^{-1/2}(\widetilde{p}_{h}-p_{H})||_{L^{2}(\Omega)}^{2}&=&
(A^{-1}(\widetilde{p}_{h}-p_{H}),\widetilde{p}_{h}-p_{H})_{0,\Omega}\vspace{2mm}\\
&=&(A^{-1}\widetilde{p}_{h},\widetilde{p}_{h}-p_{H})_{0,\Omega}-
(A^{-1}p_{H},\widetilde{p}_{h}-p_{H})_{0,\Omega}\vspace{2mm}\\
&=&-(A^{-1}p_{H},\widetilde{p}_{h}-p_{H})_{0,\Omega}
=-(A^{-1}p_{H},{\rm Curl}\ \psi_{h})_{0,\Omega}
\end{array}
\end{equation}

Since ${\rm div}( {\rm Curl} \psi_{H})=0$ for any $\psi_{H}\in
S_{H}^{l+2}$, from (\ref{quasi-optimality 2.3}) with $q_{H}={\rm
Curl}\ \psi_{H}$, we have
\begin{equation}\label{Theorem 6.1.5}
(A^{-1}p_{H},{\rm Curl}\ \psi_{H})_{0,\Omega}=-({\rm div}\ {\rm
Curl}\ \psi_{H},u_{H})_{0,\Omega}=0\ \ {\rm for\ all}\ \ \psi_{H}\in
S_{H}^{l+2}.
\end{equation}
To connect $S_{h}^{l+2}$ with $S_{H}^{l+2}$, we need to use some
local quasi-interpolation, e.g. the Scott-Zhang interpolation,
\cite{Scott} $\mathcal{I}_{H} : S_{h}^{l+2}\rightarrow S_{H}^{l+2}$ with
\begin{equation}\label{Theorem 6.1.6}
||\psi_{h}-\mathcal{I}_{H}\psi_{h}||_{L^{2}(E)}\lesssim
H_{E}^{1/2}|\psi_{h}|_{H^{1}(\omega_{E})}\ \ {\rm for\ all}\ \
E\in\varepsilon_{H}
\end{equation}
and
\begin{equation}\label{Theorem 6.1.7}
||\psi_{h}-\mathcal{I}_{H}\psi_{h}||_{L^{2}(T)}\lesssim
H_{T}|\psi_{h}|_{H^{1}(\omega_{T})}\ \ {\rm for\ all}\ \
T\in\mathcal{T}_{H}.
\end{equation}
We note   the quasi-interpolation $\mathcal{I}_{H}$ is local in
the sense that if $T\in\mathcal{T}_{h}\cap\mathcal{T}_{H}$ or
$E\in\varepsilon_{h}\cap\varepsilon_{H}$ (i.e., $T$ or $E$ is not
refined), then $(\psi_{h}-\mathcal{I}_{H}\psi_{h})|_{T}=0$ or
$(\psi_{h}-\mathcal{I}_{H}\psi_{h})|_{E}=0$.

With $\mathcal{F}_{H}$ defined in Theorem \ref{thm 6.1} and
$\psi_{H}=\mathcal{I}_{H}\psi_{h}$, by using integration by parts, a
combination of (\ref{Theorem 6.1.4}) and (\ref{Theorem 6.1.5})
yields
\begin{equation}\label{Theorem 6.1.8}
\begin{array}{lll}
||A^{-1/2}(\widetilde{p}_{h}-p_{H})||_{L^{2}(\Omega)}^{2}&=&
-(A^{-1}p_{H},{\rm
curl}(\psi_{h}-\psi_{H}))_{0,\Omega}\vspace{2mm}\\
&=&\displaystyle\sum\limits_{T\in\mathcal{T}_{H}}-\int_{T}A^{-1}p_{H}\cdot{\rm
curl}(\psi_{h}-\psi_{H})\vspace{2mm}\\
&=&\displaystyle\sum\limits_{T\in\mathcal{T}_{H}}\int_{T}{\rm
curl}(A^{-1}p_{H})(\psi_{h}-\psi_{H})\vspace{2mm}\\
&\ &\ -\displaystyle\sum\limits_{E\in\varepsilon_{H}}
\int_{E}J(A^{-1}p_{H}\cdot\tau)(\psi_{h}-\psi_{H}).
\end{array}
\end{equation}
Applying (\ref{Theorem 6.1.6}) and (\ref{Theorem 6.1.7}) to the
identity (\ref{Theorem 6.1.8}), we arrive at
\begin{equation*}
\begin{array}{lll}
||A^{-1/2}(\widetilde{p}_{h}-p_{H})||_{L^{2}(\Omega)}^{2}&\lesssim&
\eta_{\mathcal{T}_{H}}(p_{H},\mathcal{F}_{H})|\psi_{h}|_{H^{1}(\Omega)}\vspace{2mm}\\
&\leq&C_{1}^{1/2}\eta_{\mathcal{T}_{H}}(p_{H},\mathcal{F}_{H})
||A^{-1/2}(\widetilde{p}_{h}-p_{H})||_{L^{2}(\Omega)},
\end{array}
\end{equation*}
which yields
\begin{equation}\label{Theorem 6.1.10}
||A^{-1/2}(\widetilde{p}_{h}-p_{H})||_{L^{2}(\Omega)}\leq
C_{1}^{1/2}\eta_{\mathcal{T}_{H}}(p_{H},\mathcal{F}_{H}).
\end{equation}
On the other hand,
from $\widetilde{p}_{h}-p_{H}\in M_{h}$ we have
\begin{equation}\label{Theorem 6.1.9}
(A^{-1}(p_{h}-\widetilde{p}_{h}),\widetilde{p}_{h}-p_{H})_{0,\Omega}=-
({\rm
div}(\widetilde{p}_{h}-p_{H}),u_{h}-\widetilde{u}_{h})_{0,\Omega}=0.
\end{equation}
Then a combination of (\ref{Theorem 6.1.9}), (\ref{Theorem 6.1.10}) and
Lemma 
\ref{lem 4.4} yields
\begin{equation*}
\begin{array}{lll}
||A^{-1/2}(p_{h}-p_{H})||_{L^{2}(\Omega)}^{2}&=&
(A^{-1}(p_{h}-p_{H}),p_{h}-p_{H})_{0,\Omega}\vspace{2mm}\\
&=&(A^{-1}(p_{h}-\widetilde{p}_{h}+\widetilde{p}_{h}-p_{H}),p_{h}-
\widetilde{p}_{h}+\widetilde{p}_{h}-p_{H})_{0,\Omega}\vspace{2mm}\\
&=&||A^{-1/2}(p_{h}-\widetilde{p}_{h})||_{L^{2}(\Omega)}^{2}+
||A^{-1/2}(\widetilde{p}_{h}-p_{H})||_{L^{2}(\Omega)}^{2}\vspace{2mm}\\
&\leq&C_{1}\eta_{\mathcal{T}_{H}}^{2}(p_{H},\mathcal{F}_{H})+C_{0}{\rm
osc}^{2}(f_{h},\mathcal{T}_{H}),
\end{array}
\end{equation*}
namely the   result (\ref{Theorem 6.1.1}) holds.
\end{proof}

\begin{remark}\label{rem 6.1}  One can also include the second term, ${\rm
osc}^{2}(f_{h},\mathcal{T}_{H})$, of the right-hand side in (\ref{Theorem
6.1.1})  in the first term
$\eta_{\mathcal{T}_{H}}^{2}(p_{H},\mathcal{F}_{H})$. In doing so, however, one
cannot expect any relaxation of   complexity of analysis, because the oscillation of
$f_{h}$ over $\mathcal{T}_{H}$ still appears in the contraction property
(see Theorem \ref{thm 5.1}).
\end{remark}

\begin{corollary}\label{cor 6.2}({\rm Discrete upper bound})  Under
the assumption of Theorem \ref{thm 6.1}, there exist constants $C_{1}$ and
$C_{0}$ depending only on the shape regularity of $\mathcal{T}_{H}$
such that
\begin{equation}\label{Corollary 6.2}
||A^{-\frac{1}{2}}(p_{h}-p_{H})||_{L^{2}(\Omega)}^{2}+||h{\rm
div}(p_{h}-p_{H})||_{L^{2}(\Omega)}^{2}\leq
C_{1}\eta_{\mathcal{T}_{H}}^{2}(p_{H},\mathcal{F}_{H})+C_{0}{\rm
osc}^{2}(f,\mathcal{T}_{H}).
\end{equation}
\end{corollary}
\begin{proof}
Since
\begin{equation}\label{Corollary 6.2.1}
||h{\rm div}(p_{h}-p_{H})||_{L^{2}(\Omega)}^{2}\leq||H{\rm
div}(p_{h}-p_{H})||_{L^{2}(\Omega)}^{2}={\rm
osc}^{2}(f_{h},\mathcal{T}_{H}),
\end{equation}
in view of (\ref{Theorem 6.1.1}) and 
 Lemma \ref{lem 4.6}, we obtain the desired result (\ref{Corollary 6.2}).
\end{proof}

Note that the constant $C_{0}$ in Corollary \ref{cor 6.2} is actually the
constant $C_{0}$ appeared in Theorem \ref{thm 6.1} plus $1$. Here, for simplicity we still
denote it by $C_{0}$.

In what follows we shall prove a stable result in the continuous level. Let $f_{H}$ denote the
$L^{2}-$projection of $f$ over $L_{H}$, we consider the following
problem:
\begin{equation}\label{quasi-optimality2.5}
 \left \{ \begin{array}{ll}
  -\mbox{div}(A\nabla \tilde{u})=f_{H}\ \ \  & \mbox{in}\ \ \Omega\\
 \ \hspace{17mm}\tilde{u}=0 & \mbox{on}\ \ \ \partial{\Omega}.
 \end{array}\right.
\end{equation}
By the Lax-Milgram lemma, there exists unique solution $\tilde{u}\in
H_{0}^{1}(\Omega)$ to the problem (\ref{quasi-optimality2.5}).

\begin{lemma}\label{lem 6.3}({\rm Stable result})  Given a shape
regular triangulation $\mathcal{T}_{H}$ of $\Omega$, let $u\in H_{0}^{1}(\Omega)$ and
$\widetilde{u}\in H_{0}^{1}(\Omega)$ are respectively the weak solutions to the problems
 (\ref{quasi-optimality 1}) and
(\ref{quasi-optimality2.5}), and let $p=A\nabla u$ and
$\widetilde{p}=A\nabla\widetilde{u}$ denote the continuous flux.
Then there exists a positive constant $C_{0}$ depending only on the shape
regularity of $\mathcal{T}_{H}$ and the eigenvalues of $A$ such that
\begin{equation}\label{Lemma 6.3}
||A^{-1/2}(p-\widetilde{p})||_{L^{2}(\Omega)}\leq C_{0}^{1/2}{\rm
osc}(f,\mathcal{T}_{H}.)
\end{equation}
\end{lemma}
\begin{proof} From Green's
formula we have
\begin{eqnarray}\label{Lemma 6.3.1}
||A^{-1/2}(p-\widetilde{p})||_{L^{2}(\Omega)}^{2}&=&
||A^{1/2}\nabla(u-\widetilde{u})||_{L^{2}(\Omega)}^{2}\nonumber\\
&=&-\displaystyle\int_{\Omega}(u-\widetilde{u})\nabla\cdot(A\nabla(u-\widetilde{u}))\nonumber\\
&=&(f-f_{H},u-\widetilde{u})_{0,\Omega}.
\end{eqnarray}
Let $V_{H}^{c}\subset
L_{H}$ be a conforming finite element space, and set $\overline{\omega}_{T} :=\bigcup\{
T'\in\mathcal{T}_{H}: T'\cap \overline{T}\neq\emptyset\}$ for $T\in\mathcal{T}_{H}$.
 We consider the
Cl\'{e}ment or Scott-Zhang interpolation operator or other
regularized conforming
 finite element approximation operator
$\mathcal{J} :H_{0}^{1}(\Omega)\rightarrow V_{H}^{c}$ which satisfies
\begin{equation}\label{Lemma 6.3.2}
H_{T}^{-1}||v-\mathcal{J}v||_{L^{2}(T)}\lesssim ||\nabla
v||_{L^{2}(\overline{\omega}_{T})}\ \ {\rm for\ all}\
T\in\mathcal{T}_{H},v\in H_{0}^{1}(\Omega).
\end{equation}

Existence of such an operator is guaranteed (cf.
\cite{Clement,Scott,Carstensen,Carstensen;Bartels}).  Recall that
$f_{H} :=\Pi_{L_{H}}f$ is $L^{2}-$projection of $f$ over $L_{H}$.
This means $(f-f_{H},v_{H})_{0,\Omega}=0$ for all $v_{H}\in L_{H}$.
Denoting $v:=u-\widetilde{u}$, from (\ref{Lemma 6.3.1}) and (\ref{Lemma 6.3.2}) we  arrive at
\begin{equation*}
\begin{array}{lll}
||A^{-1/2}(p-\widetilde{p})||_{L^{2}(\Omega)}^{2}&=&
(f-f_{H},v-\mathcal{J}v)_{0,\Omega}\vspace{2mm}\\
&=&\displaystyle\sum\limits_{T\in\mathcal{T}_{H}}(f-f_{H},v-\mathcal{J}v)_{0,T}\vspace{2mm}\\
&\lesssim&\displaystyle\sum\limits_{T\in\mathcal{T}_{H}}||H(f-f_{H})||_{L^{2}(T)}||\nabla
v||_{L^{2}(T)}\vspace{2mm}\\
&\leq& C_{0}^{1/2}{\rm
osc}(f,\mathcal{T}_{H})||A^{-1/2}(p-\widetilde{p})||_{L^{2}(\Omega)},
\end{array}
\end{equation*}
which implies the desired result (\ref{Lemma 6.3}).
\end{proof}

We finish this subsection by quoting a counting conclusion from
\cite{Stevenson} for the
overlay $\mathcal{T} :=\mathcal{T}_{1}\oplus\mathcal{T}_{2}$ of two conforming
triangulations $\mathcal{T}_{1}$ and $\mathcal{T}_{2}$, which shows 
$\mathcal{T}$ is the smallest conforming triangulation for the triangulations $\mathcal{T}_{1}$ and $\mathcal{T}_{2}$. 

\begin{lemma}\label{lem 6.4}({\rm Overlay of meshes})  For two
conforming triangulations $\mathcal{T}_{1}$ and $\mathcal{T}_{2}$
the overlay $\mathcal{T} :=\mathcal{T}_{1}\oplus\mathcal{T}_{2}$ is
conforming, and satisfies
\begin{equation*}
\#\mathcal{T}\leq\#\mathcal{T}_{1}+\#\mathcal{T}_{2}-\#\mathcal{T}_{0}.
\end{equation*}
\end{lemma}

\subsection{Quasi-optimal convergence rate}
In this subsection, we shall prove the quasi-optimal convergence
rate of the AMFEM for the stress variable error in a weighted norm.
To this end, we need to introduce two nonlinear approximation
classes. Let $\mathcal{P}_{N}$ be the set of all triangulations
$\mathcal{T}$ which is refined from $\mathcal{T}_{0}$ and
$\#\mathcal{T}\leq N$. For a given triangulation $\mathcal{T}$, let
$M_{\mathcal{T}}$ and $L_{\mathcal{T}}$ denote respectively the approximation
spaces to the flux and displacement,
$(p_{\mathcal{T}},u_{\mathcal{T}})\in M_{\mathcal{T}}\times
L_{\mathcal{T}}$ be the approximation to $(p,u)$, and $h_{\mathcal{T}}$ be mesh-size
functions with respect to the triangulation $\mathcal{T}$. The quantity
of the best approximation to the total error in $\mathcal{P}_{N}$ is
given by
\begin{equation*}
\begin{array}{lll}
\sigma(N;p,f,A) :&=&\inf_{\mathcal{T}\in\mathcal{P}_{N}}\{
||A^{-1/2}(p-p_{\mathcal{T}})||_{L^{2}(\Omega)}^{2}\vspace{2mm}\\
&\ &\ +||h_{\mathcal{T}}{\rm
div}(p-p_{\mathcal{T}})||_{L^{2}(\Omega)}^{2}+{\rm
osc}_{\mathcal{T}}^{2}(p_{\mathcal{T}},\mathcal{T})\}^{1/2},
\end{array}
\end{equation*}
and for $s>0$ we define the nonlinear approximation class
$\mathbb{A}_{s}$ as
\begin{equation*}
\mathbb{A}_{s} :=\{(p,f,A)|\ |(p,f,A)|_{s} :=
\sup_{N>N_{0}=\#\mathcal{T}_{0}}N^{s}\sigma(N;p,f,A)<\infty\}.
\end{equation*}
Moreover , the quantity of the best approximation to the right-hand
side term $f$ in $\mathcal{P}_{N}$ is described by
\begin{equation*}
||f||_{\mathcal{A}_{0}^{s}}
:=\sup_{N>N_{0}=\#\mathcal{T}_{0}}N^{s}\inf_{\mathcal{T}\in\mathcal{P}_{N}}
{\rm osc}(f,\mathcal{T}).
\end{equation*}
By the nonlinear approximation theory \cite{DeVore1,DeVore2}, we
know that if $f\in L^{2}(\Omega)$ then
$||f||_{\mathbb{A}_{0}^{s}}<\infty$. Here, we recall a result of
Binev, Dahmen, and DeVore \cite{Binev;Dahmen;DeVore} which shows
that the approximation of data $f$ can be done in an optimal way.
The proof of this result can be found in
\cite{Binev;Dahmen;DeVore,Binev;Dahmen;DeVore;Petrushev}.

\begin{lemma}\label{lem 6.5}({\rm Approximation of data $f$})  Given an $f\in L^{2}(\Omega)$, a
tolerance $\varepsilon$,  and a shape regular
triangulation $\mathcal{T}_{0}$, there exists an algorithm
\begin{equation*}
\mathcal{T}_{H}=APPROX(f,\mathcal{T}_{0},\varepsilon)
\end{equation*}
such that
\begin{equation*}
{\rm osc}(f,\mathcal{T}_{H})\leq\varepsilon,\ \ \ \ \
\#\mathcal{T}_{H}-\#\mathcal{T}_{0}\leq
C||f||_{\mathcal{A}_{0}^{s}}^{1/s} \varepsilon^{-1/s}.
\end{equation*}
\end{lemma}

We now prove that the approximation $p_{k}$ generated by the AMFEM
concerning the stress variable converges to $p$ in a weighted norm
with the same rate $(\#\mathcal{T}_{k}-\#\mathcal{T}_{0})^{-s}$ as
the best approximation described by $\mathbb{A}_{s}$ up to a
multiplicative constant. We need to count elements added by handling
hanging nodes to keep mesh conformity (see Lemma 
\ref{lem 2.2}), as well as
those marked by the estimator (the cardinality of
$\mathcal{M}_{k}$). To this end, we impose more stringent
requirements than for convergence of the AMFEM.\\

\noindent\textbf{Assumption\ 6.1 }({\rm Optimality}).\ {\it we assume the
following properties of the AMFEM:\\
(a)\ The marking parameter $\theta$ satisfies
$\theta\in(0,\theta_{*})$ with
\begin{equation*}
\theta_{*}^{2}=\frac{C_{2}}{1+C_{1}(1+2\Lambda_{1}{\rm
osc}_{\mathcal{T}_{0}}^{2}(A^{-1},\mathcal{T}_{0}))};
\end{equation*}
(b)\ Procedure $MARK$ selects a set $\mathcal{M}_{k}$ of marked
elements with minimal cardinality;\\
(c)\ The distribution of refinement edges on $\mathcal{T}_{0}$
satisfies condition (b) of section 4 in \cite{Stevenson1}. }

The limit value $\theta_{*}$ depends on the ratio
$(C_{2}/C_{1})^{1/2}\leq1$, which quantifies the quality of
approximation to the stress variable of estimator
$\eta_{\mathcal{T}_{k}}(p_{k},\mathcal{T}_{k})$, as well as the
oscillation ${\rm osc}_{\mathcal{T}_{0}}(A^{-1},\mathcal{T}_{0})$ of
coefficient matrix of the PDEs over $\mathcal{T}_{0}$.

The following lemma establishes a link between nonlinear
approximation theory and the AMFEM through the D\"{o}rfler marking
strategy. Roughly speaking, we prove that, if an approximation
satisfies a suitable total error reduction from $\mathcal{T}_{H}$ to
$\mathcal{T}_{h}$ ($\mathcal{T}_{h}$ is a refinement of
$\mathcal{T}_{H}$), the error indicators of the coarser solutions
must satisfy a D\"{o}rfler property on the set $\mathcal{R}$ of
refined elements. In other words, the total error reduction and
D\"{o}rfler marking are intimately connected.

\begin{lemma}\label{lem 6.7}({\rm Optimality marking})  Assume that
the marking parameter $\theta$ verifies (a) of Assumption 6.1, and
that $f$ is a piecewise polynomial of degree $\leq l$ on
$\mathcal{T}_{H}$. Let $\mathcal{T}_{H}$ be an shape regular
triangulation of $\Omega$, $(p_{H},u_{H})\in M_{H}\times L_{H}$ be a
pair of discrete solutions of (\ref{quasi-optimality 2.3}). Set $\mu
:=\frac{1}{2}(1-\frac{\theta^{2}}{\theta_{*}^{2}})>0$, and let
$\mathcal{T}_{h_{*}}$ be any refinement of $\mathcal{T}_{H}$ such
that a pair of discrete solutions $(p_{h_{*}},u_{h_{*}})\in
M_{h_{*}}\times L_{h_{*}}$ satisfies
\begin{equation}\label{Lemma 6.7}
\mathcal{E}_{h_{*}}^{2}+{\rm
osc}_{\mathcal{T}_{h_{*}}}^{2}(p_{h_{*}},\mathcal{T}_{h_{*}})\leq\mu\{
\mathcal{E}_{H}^{2}+{\rm
osc}_{\mathcal{T}_{H}}^{2}(p_{H},\mathcal{T}_{H})\}.
\end{equation}
Then the set $\mathcal{R}
:=\mathcal{R}_{\mathcal{T}_{H}\rightarrow\mathcal{T}_{h_{*}}}$
satisfies the D\"{o}rfler property
\begin{equation*}
\eta_{\mathcal{T}_{H}}(p_{H},\mathcal{R})\geq
\theta\eta_{\mathcal{T}_{H}}(p_{H},\mathcal{T}_{H}).
\end{equation*}
\end{lemma}
\begin{proof} Since $f$ is
a piecewise polynomial of degree $\leq l$ on $\mathcal{T}_{H}$, it
holds $${\rm osc}(f,\mathcal{T}_{H})=0 \ \ \text{ and }\ \ {\rm
osc}(f_{h_{*}},\mathcal{T}_{h_{*}})=0,$$
which imply  
$$||H{\rm
div}(p-p_{H})||_{L^{2}(\Omega)}=0\ \ \text{ and }\ \  ||h_{*}{\rm
div}(p-p_{h_{*}})||_{L^{2}(\Omega)}=0.$$ 
These two relations, together with   the lower bound (\ref{Lemma 3.5}),
the condition (\ref{Lemma 6.7}) and  the quasi-orthogonality
(\ref{Lemma 4.5.2}), yield
\begin{equation}\label{Lemma 6.7.1}
\begin{array}{lll}
(1-2\mu)C_{2}\eta_{\mathcal{T}_{H}}^{2}(p_{H},\mathcal{T}_{H})&\leq&(1-2\mu)
\{\mathcal{E}_{H}^{2}+{\rm
osc}_{\mathcal{T}_{H}}^{2}(p_{H},\mathcal{T}_{H})\}\vspace{2mm}\\
&\leq&||A^{-1/2}(p-p_{H})||_{L^{2}(\Omega)}^{2}-
||A^{-1/2}(p-p_{h_{*}})||_{L^{2}(\Omega)}^{2}\vspace{2mm}\\
&&+||H{\rm div}(p-p_{H})||_{L^{2}(\Omega)}^{2}-||h_{*}{\rm
div}(p-p_{h_{*}})||_{L^{2}(\Omega)}^{2}\vspace{2mm}\\
&&{+\rm osc}_{\mathcal{T}_{H}}^{2}(p_{H},\mathcal{T}_{H})-2{\rm
osc}_{\mathcal{T}_{h_{*}}}^{2}(p_{h_{*}},\mathcal{T}_{h_{*}}).
\end{array}
\end{equation}

We estimate separately the error and oscillation terms. By the quasi-orthogonality
(\ref{Lemma 4.5.2}) and discrete upper bound (\ref{Theorem 6.1.1}),
we get
\begin{equation}\label{Lemma 6.7.2}
\begin{array}{lll}
||A^{-1/2}(p-p_{H})||_{L^{2}(\Omega)}^{2}&-&
||A^{-1/2}(p-p_{h_{*}})||_{L^{2}(\Omega)}^{2}+||H{\rm
div}(p-p_{H})||_{L^{2}(\Omega)}^{2}\vspace{2mm}\\
&-&||h_{*}{\rm div}(p-p_{h_{*}})||_{L^{2}(\Omega)}^{2}=
||A^{-1/2}(p_{h_{*}}-p_{H})||_{L^{2}(\Omega)}^{2}\vspace{2mm}\\
&\leq& C_{1}\eta_{\mathcal{T}_{H}}^{2}(p_{H},\mathcal{R})
\end{array}
\end{equation}

For the oscillation term we argue according to whether an element
$T\in\mathcal{T}_{H}$ belongs to the set of refined elements
$\mathcal{R}$ or not. For $T\in\mathcal{R}$ we use the dominance
${\rm
osc}_{\mathcal{T}_{H}}^{2}(p_{H},T)\leq\eta_{\mathcal{T}_{H}}^{2}(p_{H},T)$
 (see Remark 
\ref{rem 2.1}). For $T\in\mathcal{T}_{H}\cap\mathcal{T}_{h_{*}}$,
Corollary \ref{cor 4.11} (Perturbation of oscillation) together with
$\sigma_{H}=p_{H}$ and $\sigma_{h}=p_{h_{*}}$ yields
\begin{equation}\label{Lemma 6.7.3}
\begin{array}{lll}
&\ &{\rm
osc}_{\mathcal{T}_{H}}^{2}(p_{H},\mathcal{T}_{H}\cap\mathcal{T}_{h_{*}})-2{\rm
osc}_{\mathcal{T}_{h_{*}}}^{2}(p_{h_{*}},\mathcal{T}_{H}\cap\mathcal{T}_{h_{*}})\vspace{2mm}\\
&\ &\ \ \leq2\Lambda_{1}{\rm
osc}_{\mathcal{T}_{0}}^{2}(A^{-1},\mathcal{T}_{0})||A^{-1/2}(p_{h_{*}}-p_{H})||_{L^{2}(\Omega)}^{2}.
\end{array}
\end{equation}

Combining (\ref{Lemma 6.7.2}) and (\ref{Lemma 6.7.3}) we infer that
\begin{equation}\label{Lemma 6.7.4}
\begin{array}{lll}
&\ &{\rm osc}_{\mathcal{T}_{H}}^{2}(p_{H},\mathcal{T}_{H})-2{\rm
osc}_{\mathcal{T}_{h_{*}}}^{2}(p_{h_{*}},\mathcal{T}_{h_{*}})\vspace{2mm}\\
&\ &={\rm osc}_{\mathcal{T}_{H}}^{2}(p_{H},\mathcal{R})+{\rm
osc}_{\mathcal{T}_{H}}^{2}(p_{H},\mathcal{T}_{H}\cap\mathcal{T}_{h_{*}})\vspace{2mm}\\
&\ &\ \ \ \ -2{\rm
osc}_{\mathcal{T}_{h_{*}}}^{2}(p_{h_{*}},\mathcal{T}_{H}\cap\mathcal{T}_{h_{*}})-2{\rm
osc}_{\mathcal{T}_{h_{*}}}^{2}(p_{h_{*}},\mathcal{T}_{h_{*}}\setminus\mathcal{T}_{H})\vspace{2mm}\\
&\ &\leq{\rm osc}_{\mathcal{T}_{H}}^{2}(p_{H},\mathcal{R})+2\Lambda_{1}{\rm
osc}_{\mathcal{T}_{0}}^{2}(A^{-1},\mathcal{T}_{0})||A^{-1/2}(p_{h_{*}}-p_{H})||_{L^{2}(\Omega)}^{2}\vspace{2mm}\\
&\ &\leq(1+2C_{1}\Lambda_{1}{\rm
osc}_{\mathcal{T}_{0}}^{2}(A^{-1},\mathcal{T}_{0}))\eta_{\mathcal{T}_{H}}^{2}(p_{H},\mathcal{R}).
\end{array}
\end{equation}
From (\ref{Lemma 6.7.1}), (\ref{Lemma 6.7.2}), and (\ref{Lemma
6.7.4}), we finally deduce that
\begin{equation*}
\eta_{\mathcal{T}_{H}}^{2}(p_{H},\mathcal{R})\geq\frac{(1-2\mu)C_{2}}{1+C_{1}(1+2\Lambda_{1}
{\rm
osc}_{\mathcal{T}_{0}}^{2}(A^{-1},\mathcal{T}_{0}))}\eta_{\mathcal{T}_{H}}^{2}(p_{H},\mathcal{T}_{H})=
\theta^{2}\eta_{\mathcal{T}_{H}}^{2}(p_{H},\mathcal{T}_{H}).
\end{equation*}
In light of the definitions of $\theta_{*},\theta$, and $\mu$, this
concludes the proof.
\end{proof}

The fact that procedure MARK selects the set of marked elements
$\mathcal{M}_{k}$ with minimal cardinality, establishes a link
between the best mesh and triangulations generated by AMFEM, and
forms crucial idea of AFEM (see \cite{Stevenson}). In what follows
we shall use this fact.

\begin{lemma}\label{lem 6.8}({\rm Cardinality of $\mathcal{M}_{k}$}) 
Assume that the marking parameter $\theta$ verifies (a) of
Assumption 6.1, and procedure MARK satisfies (b) of Assumption 6.1,
and that $f$ is the piecewise polynomial of degree $\leq l$ onto
$\mathcal{T}_{0}$. Let $(p,u)$ solve the problem
(\ref{quasi-optimality 2.1}), and let
$\{\mathcal{T}_{k};(M_{k},L_{k});(p_{k},u_{k});\mathcal{E}_{k}\}_{k\geq0}$
be the sequence of meshes, finite element spaces, the discrete
solution
produced by the AMFEM, and the stress variable error in weighted norm\\
If $(p,f,A)\in\mathbb{A}_{s}$, then the following estimate is
valid:
\begin{equation}\label{Lemma 6.8.1}
\begin{array}{lll}
\#\mathcal{M}_{k}&\lesssim&(1-\frac{\theta^{2}}{\theta_{*}^{2}})^{-1/2s}
|(p,f,A)|_{s}^{1/s}C_{A}^{1/2s}\{\mathcal{E}_{k}^{2}+ {\rm
osc}_{\mathcal{T}_{k}}^{2}(p_{k},\mathcal{T}_{k})\}^{-1/2s}.
\end{array}
\end{equation}
Furthermore, if $\mathcal{T}_{k+1}$ is a refinement of
$\mathcal{T}_{k}$ obtained by the algorithm $REFINE$ with
$\theta\in(0,\theta_{*})$, then
\begin{equation}\label{Lemma 6.8.2}
\begin{array}{lll}
\#\mathcal{T}_{k+1}-\#\mathcal{T}_{k}\lesssim
(1-\frac{\theta^{2}}{\theta_{*}^{2}})^{-1/2s}
|(p,f,A)|_{s}^{1/s}C_{A}^{1/2s} \{\mathcal{E}_{k}^{2}+{\rm
osc}_{\mathcal{T}_{k}}^{2}(p_{k},\mathcal{T}_{k})\}^{-1/2s}.
\end{array}
\end{equation}
\end{lemma}
\begin{proof} Let $[\varepsilon^{-1/s}|(p,f,A)|_{s}^{1/s}]$ denote the integer
component of $\varepsilon^{-1/s}|(p,f,A)|_{s}^{1/s}$. We set
$\varepsilon^{2} :=4^{-1}C_{A}^{-1}\mu(\mathcal{E}_{k}^{2}+{\rm
osc}_{\mathcal{T}_{k}}^{2}(p_{k},\mathcal{T}_{k}))$, where
$\mu=\frac{1}{2}(1-\frac{\theta^{2}}{\theta_{*}^{2}})>0$ and
$C_{A}=\max\{1+2\Lambda_{1}^{2}{\rm
osc}_{\mathcal{T}_{0}}^{2}(A^{-1},\mathcal{T}_{0}),2\}$, and set
$N_{\varepsilon} :=[\varepsilon^{-1/s}|(p,f,A)|_{s}^{1/s}]+1$.
Recall that
\begin{equation*}
\sigma(N_{\varepsilon}+\#\mathcal{T}_{0}-1;p,f,A)
:=\inf_{\mathcal{T}\in\mathcal{P}_{N_{\varepsilon}+
\#\mathcal{T}_{0}-1}}\{\mathcal{E}_{\mathcal{T}}^{2}+{\rm
osc}_{\mathcal{T}}^{2}(p_{\mathcal{T}}, \mathcal{T})\}^{1/2},
\end{equation*}
where $\mathcal{E}_{\mathcal{T}}^{2}
:=||A^{-1/2}(p-p_{\mathcal{T}})||_{L^{2}(\Omega)}^{2}
+||h_{\mathcal{T}}{\rm
div}(p-p_{\mathcal{T}})||_{L^{2}(\Omega)}^{2}$,\\
and
\begin{equation*}
\mathcal{E}_{h_{\varepsilon}}^{2}
:=||A^{-1/2}(p-p_{\varepsilon})||_{L^{2}(\Omega)}^{2}+
||h_{\varepsilon}{\rm div}(p-p_{\varepsilon})||_{L^{2}(\Omega)}^{2},
\end{equation*}
where $p_{\varepsilon}$ is the discrete flux approximation to $p
:=A\nabla u$ with respect to the mesh $\mathcal{T}_{\varepsilon}$,
and $h_{\varepsilon}$ is the mesh-size function with respect to
$\mathcal{T}_{\varepsilon}$.

Since there exists
$\mathcal{T}_{\varepsilon}\in\mathcal{P}_{N_{\varepsilon}+\#\mathcal{T}_{0}-1}$
with $\#\mathcal{T}_{\varepsilon}\leq
N_{\varepsilon}+\#\mathcal{T}_{0}-1$ such that
\begin{equation*}
\{\mathcal{E}_{h_{\varepsilon}}^{2}+{\rm
osc}_{\mathcal{T}_{\varepsilon}}^{2}(p_{\varepsilon},
\mathcal{T}_{\varepsilon})\}^{1/2}\leq(1+\tilde{\varepsilon})\sigma(N_{\varepsilon}
+\#\mathcal{T}_{0}-1;p,f,A),
\end{equation*}
where $\tilde{\varepsilon} :=\min\{1,\varepsilon\}$. This inequality
leads to
\begin{equation}\label{Lemma 6.8.3}
\begin{array}{lll}
&\ &\ N_{\varepsilon}^{s}\{\mathcal{E}_{h_{\varepsilon}}^{2}+{\rm
osc}_{\mathcal{T}_{\varepsilon}}^{2}(p_{\varepsilon},
\mathcal{T}_{\varepsilon})\}^{1/2}\vspace{2mm}\\
&\ &\ \ \leq(N_{\varepsilon}+\#\mathcal{T}_{0}-1)^{s}
\{\mathcal{E}_{h_{\varepsilon}}^{2}+{\rm
osc}_{\mathcal{T}_{\varepsilon}}^{2}(p_{\varepsilon},
\mathcal{T}_{\varepsilon})\}^{1/2}\vspace{2mm}\\
&\ &\ \
\leq(1+\tilde{\varepsilon})(N_{\varepsilon}+\#\mathcal{T}_{0}-1)^{s}\sigma(N_{\varepsilon}
+\#\mathcal{T}_{0}-1;p,f,A)\vspace{2mm}\\
&\ &\ \
\leq(1+\tilde{\varepsilon})\sup_{N>0}N^{s}\sigma(N;p,f,A)=(1+\tilde{\varepsilon})|(p,f,A)|_{s}.
\end{array}
\end{equation}

From the above inequality (\ref{Lemma 6.8.3}), we obtain
\begin{equation}\label{Lemma 6.8.4}
\{\mathcal{E}_{h_{\varepsilon}}^{2}+{\rm
osc}_{\mathcal{T}_{\varepsilon}}^{2}(p_{\varepsilon},
\mathcal{T}_{\varepsilon})\}^{1/2}\leq
\frac{(1+\tilde{\varepsilon})|(p,f,A)|_{s}}{N_{\varepsilon}^{s}}\leq2\varepsilon
\end{equation}
and
\begin{equation}\label{Lemma 6.8.5}
\#\mathcal{T}_{\varepsilon}-\#\mathcal{T}_{0}\leq N_{\varepsilon}-1
\leq\varepsilon^{-1/s}|(p,f,A)|_{s}^{1/s}.
\end{equation}

Let $\mathcal{T}_{*}
:=\mathcal{T}_{\varepsilon}\oplus\mathcal{T}_{k}$ be the overlay of
$\mathcal{T}_{\varepsilon}$ and $\mathcal{T}_{k}$, $h_{*}$ denote
the mesh-size function with respect to $\mathcal{T}_{*}$, and
$(p_{*},u_{*})$ be a pair of discrete solutions onto
$\mathcal{T}_{*}$. We shall show that there is a reduction with a
factor $\mu$ of the total error between $p_{*}$ and $p_{k}$. Notice
that $\mathcal{T}_{*}$ is a refinement of
$\mathcal{T}_{\varepsilon}$, and since $f$ is the piecewise
polynomial of degree $\leq l$ onto $\mathcal{T}_{0}$. Recall that
$\mathcal{E}_{h_{*}}^{2}
:=||A^{-1/2}(p-p_{*})||_{L^{2}(\Omega)}^{2}+ ||h_{*}{\rm
div}(p-p_{*})||_{L^{2}(\Omega)}^{2}$, by the quasi-orthogonality
(\ref{Lemma 4.5.2}) and ${\rm div}(p_{*}-p_{\varepsilon})=0$, we
get
\begin{equation}\label{Lemma 6.8.6}
\begin{array}{lll}
&\ &\mathcal{E}_{h_{*}}^{2}+{\rm osc}_{\mathcal{T}_{*}}^{2}(p_{*},
\mathcal{T}_{*})=||A^{-1/2}(p-p_{\varepsilon})||_{L^{2}(\Omega)}^{2}-
||A^{-1/2}(p_{*}-p_{\varepsilon})||_{L^{2}(\Omega)}^{2}\vspace{2mm}\\
&\ &\ \ +||h_{*}{\rm
div}(p-p_{\varepsilon})||_{L^{2}(\Omega)}^{2}-||h_{*}{\rm
div}(p_{*}-p_{\varepsilon})||_{L^{2}(\Omega)}^{2}+{\rm
osc}_{\mathcal{T}_{*}}^{2}(p_{*}, \mathcal{T}_{*})\vspace{2mm}\\
&\ &\ \ \leq \mathcal{E}_{h_{\varepsilon}}^{2}+{\rm
osc}_{\mathcal{T}_{*}}^{2}(p_{*}, \mathcal{T}_{*}).
\end{array}
\end{equation}

By the second inequality (\ref{Lemma 4.9.2}) of Lemma \ref{lem 4.9} with
$p_{\varepsilon},p_{*}\in M_{h_{*}}$, for all $T\in\mathcal{T}_{*}$,
we have
\begin{equation*}
{\rm osc}_{\mathcal{T}_{*}}^{2}(p_{*}, T)\leq2{\rm
osc}_{\mathcal{T}_{*}}^{2}(p_{\varepsilon},
T)+2\bar{\Lambda}_{1}^{2}{\rm osc}_{\mathcal{T}_{*}}^{2}(A^{-1},
T)||A^{-1/2}(p_{*}-p_{\varepsilon})||_{L^{2}(\omega_{T})}^{2}.
\end{equation*}
Summing on $\mathcal{T}_{*}$, the monotonicity of the data
oscillation (see Remarks 
\ref{rem 2.1} and \ref{rem 4.7}), we get
\begin{equation}
\begin{array}{lll}\label{6.8.7}
{\rm osc}_{\mathcal{T}_{*}}^{2}(p_{*}, \mathcal{T}_{*})&\leq&2{\rm
osc}_{\mathcal{T}_{*}}^{2}(p_{\varepsilon},
\mathcal{T}_{*})+2\Lambda_{1}^{2}{\rm
osc}_{\mathcal{T}_{*}}^{2}(A^{-1},
\mathcal{T}_{*})||A^{-1/2}(p_{*}-p_{\varepsilon})||_{L^{2}(\Omega)}^{2}\vspace{2mm}\\
&\leq&2{\rm osc}_{\mathcal{T}_{\varepsilon}}^{2}(p_{\varepsilon},
\mathcal{T}_{\varepsilon})+2\Lambda_{1}^{2}{\rm
osc}_{\mathcal{T}_{0}}^{2}(A^{-1},
\mathcal{T}_{0})||A^{-1/2}(p_{*}-p_{\varepsilon})||_{L^{2}(\Omega)}^{2}\vspace{2mm}\\
&\leq&2{\rm osc}_{\mathcal{T}_{\varepsilon}}^{2}(p_{\varepsilon},
\mathcal{T}_{\varepsilon})+2\Lambda_{1}^{2}{\rm
osc}_{\mathcal{T}_{0}}^{2}(A^{-1}, \mathcal{T}_{0})\vspace{2mm}\\
&\ & \times(||A^{-1/2}(p-p_{\varepsilon})||_{L^{2}(\Omega)}^{2}-
||A^{-1/2}(p-p_{*})||_{L^{2}(\Omega)}^{2}).
\end{array}
\end{equation}
A combination of (\ref{Lemma 6.8.4}), (\ref{Lemma 6.8.6}), and
(\ref{6.8.7}) yields
\begin{equation*}
\begin{array}{lll}
\mathcal{E}_{h_{*}}^{2}+{\rm osc}_{\mathcal{T}_{*}}^{2}(p_{*},
\mathcal{T}_{*})&\leq& C_{A}(\mathcal{E}_{h_{\varepsilon}}^{2}+{\rm
osc}_{\mathcal{T}_{\varepsilon}}^{2}(p_{\varepsilon},
\mathcal{T}_{\varepsilon}))\vspace{2mm}\\
&\leq&4\varepsilon^{2}C_{A}=\mu\{\mathcal{E}_{k}^{2}+{\rm
osc}_{\mathcal{T}_{k}}^{2}(p_{k}, \mathcal{T}_{k})\}.
\end{array}
\end{equation*}

Hence, we deduce from optimality marking (Lemma \ref{lem 6.7}) that the subset
$\mathcal{R}
:=\mathcal{R}_{\mathcal{T}_{k}\rightarrow\mathcal{T}_{*}}\subset\mathcal{T}_{k}$
verifies the D\"{o}rfler property (\ref{Dorfler property}) for
$\theta<\theta_{*}$. The fact that procedure MARK selects a subset
$\mathcal{M}_{k}\subset\mathcal{T}_{k}$ with minimal cardinality
satisfying the same property (\ref{Dorfler property}), and
(\ref{Lemma 6.8.5}) leads to
\begin{equation*}
\begin{array}{lll}
\#\mathcal{M}_{k}&\leq&\#\mathcal{R}\leq\#\mathcal{T}_{*}-\#\mathcal{T}_{k}\leq
\#\mathcal{T}_{\varepsilon}+\#\mathcal{T}_{k}-\#\mathcal{T}_{0}-\#\mathcal{T}_{k}\vspace{2mm}\\
&=&\#\mathcal{T}_{\varepsilon}-\#\mathcal{T}_{0}\leq|(p,f,A)|_{s}^{1/s}\varepsilon^{-1/s}\vspace{2mm}\\
&\leq&(8C_{A})^{1/2s}(1-\frac{\theta^{2}}{\theta_{*}^{2}})^{-1/2s}
|(p,f,A)|_{s}^{1/s}\{SE_{k}^{2}+{\rm
osc}_{\mathcal{T}_{k}}^{2}(p_{k},\mathcal{T}_{k})\}^{-1/2s},
\end{array}
\end{equation*}
which implies the desired result (\ref{Lemma 6.8.1}). In the third
step above, we have used the overlay of two meshes (Lemma \ref{lem 6.4}).

The second assertion (\ref{Lemma 6.8.2}) follows from
$\#\mathcal{T}_{k+1}-\#\mathcal{T}_{k}\lesssim\#\mathcal{M}_{k}$ and
the first result (\ref{Lemma 6.8.1}).
\end{proof}

\begin{theorem}\label{thm 6.9}  Assume that $f$ is a piecewise
polynomial of degree $\leq l$ onto $\mathcal{T}_{0}$, then the
algorithm AMFEM will terminate in finite steps for a given tolerance
$\varepsilon$. Furthermore, set the algorithm AMFEM terminating in
the $N-$th step, and denote by $\mathcal{T}_{N}$ the triangulation
obtained in the $N-$th step. Let $(p,f,A)\in\mathbb{A}_{s}$, and
$\Theta(s,\theta)
:=(1-\frac{\theta^{2}}{\theta_{*}^{2}})^{-1/2s}\frac{1}{1-\alpha^{1/s}}$
describes the asymptotics of the AMFEM as
$\theta\rightarrow\theta_{*},0$ or $s\rightarrow0$. Then there
exists a constant $C$, depending on data, the refinement depth $b$,
and $\mathcal{T}_{0}$, but independent of $s$, such that
\begin{equation}\label{Lemma 6.9.1}
\#\mathcal{T}_{N}-\#\mathcal{T}_{0}\lesssim
C\Theta(s,\theta)|(p,f,A)|_{s}^{1/s}\varepsilon^{-1/s}.
\end{equation}
Furthermore, it holds
\begin{equation}\label{Lemma 6.9.2}
\{\mathcal{E}_{N}^{2}+{\rm
osc}_{\mathcal{T}_{N}}^{2}(p_{N},\mathcal{T}_{N})\}^{1/2}\lesssim
C^{s}\Theta^{s}(s,\theta)|(p,f,A)|_{s}(\#\mathcal{T}_{N}-\#\mathcal{T}_{0})^{-s}.
\end{equation}
\end{theorem}
\begin{proof} By the contraction property (Theorem \ref{thm 5.1}), and the fact that $f$ is a piecewise
polynomial of degree $\leq l$ onto $\mathcal{T}_{0}$, there exists
$\alpha\in(0,1)$ such that
\begin{equation}\label{Lemma 6.9.3}
\mathcal{E}_{k+1}^{2}+
\gamma\eta_{\mathcal{T}_{k+1}}^{2}(p_{k+1},\mathcal{T}_{k+1})\leq\alpha^{2}(\mathcal{E}_{k}^{2}+
\gamma\eta_{\mathcal{T}_{k}}^{2}(p_{k},\mathcal{T}_{k})).
\end{equation}
The first assertion is a direct consequence of
the above inequality (\ref{Lemma 6.9.3}).\\
Since
\begin{equation*}
\mathcal{E}_{k}^{2}+{\rm
osc}_{\mathcal{T}_{k}}^{2}(p_{k},\mathcal{T}_{k})\approx
\mathcal{E}_{k}^{2}+
\gamma\eta_{\mathcal{T}_{k}}^{2}(p_{k},\mathcal{T}_{k})\approx
\eta_{\mathcal{T}_{k}}^{2}(p_{k},\mathcal{T}_{k}),
\end{equation*}
Combining the complexity of $REFINE$ (Lemma 
\ref{lem 2.2}) and the cardinality
(\ref{Lemma 6.8.1}) of $\mathcal{M}_{k}$, we deduce that
\begin{equation}\label{Lemma 6.9.4}
\#\mathcal{T}_{N}-\#\mathcal{T}_{0}\lesssim\sum\limits_{k=0}^{N-1}\#\mathcal{M}_{k}
\lesssim\beta\sum\limits_{k=0}^{N-1}\{\mathcal{E}_{k}^{2}+{\rm
osc}_{\mathcal{T}_{k}}^{2}(p_{k},\mathcal{T}_{k})\}^{-1/2s},
\end{equation}
where $\beta
:=(1-\frac{\theta^{2}}{\theta_{*}^{2}})^{-1/2s}C_{A}^{1/2s}|(p,f,A)|_{s}^{1/s}$.

From the lower bound (\ref{Lemma 3.5}), we infer that
\begin{equation}\label{Lemma 6.9.5}
\mathcal{E}_{k}^{2}+\gamma{\rm
osc}_{\mathcal{T}_{k}}^{2}(p_{k},\mathcal{T}_{k})\leq\mathcal{E}_{k}^{2}+
\gamma\eta_{\mathcal{T}_{k}}^{2}(p_{k},\mathcal{T}_{k})\leq
(1+\frac{\gamma}{C_{2}})\{\mathcal{E}_{k}^{2}+{\rm
osc}_{\mathcal{T}_{k}}^{2}(p_{k},\mathcal{T}_{k})\}.
\end{equation}

On the other hand, the linear rate $\alpha=\alpha(\theta)<1$ of
convergence for the quasi error implies that for $0\leq k\leq N-1$
\begin{equation}\label{Lemma 6.9.6}
\mathcal{E}_{N-1}^{2}+
\gamma\eta_{\mathcal{T}_{N-1}}^{2}(p_{N-1},\mathcal{T}_{N-1})\leq
\alpha^{2(N-1-k)}\{\mathcal{E}_{k}^{2}+
\gamma\eta_{\mathcal{T}_{k}}^{2}(p_{k},\mathcal{T}_{k})\}.
\end{equation}
We combine the above three inequalities (\ref{Lemma
6.9.4})-(\ref{Lemma 6.9.6}) to obtain
\begin{equation*}
\#\mathcal{T}_{N}-\#\mathcal{T}_{0}\lesssim\beta(1+\frac{\gamma}{C_{2}})^{1/2s}
\{\mathcal{E}_{N-1}^{2}+\gamma\eta_{\mathcal{T}_{N-1}}^{2}(p_{N-1},\mathcal{T}_{N-1})\}^{-1/2s}
\sum\limits_{k=0}^{N-1}\alpha^{k/s}.
\end{equation*}
Since $\alpha<1$, the geometric series is bounded by the constant
$s_{\theta}=1/(1-\alpha^{1/s})$. By recalling
$\eta_{\mathcal{T}_{N-1}}^{2}(p_{N-1},\mathcal{T}_{N-1})\approx
\mathcal{E}_{N-1}^{2}+\gamma\eta_{\mathcal{T}_{N-1}}^{2}(p_{N-1},\mathcal{T}_{N-1})$,
we end up with
\begin{equation}\label{Lemma 6.9.7}
\#\mathcal{T}_{N}-\#\mathcal{T}_{0}\lesssim
s_{\theta}\beta(1+\frac{\gamma}{C_{2}})^{1/2s}
\{\eta_{\mathcal{T}_{N-1}}^{2}(p_{N-1},\mathcal{T}_{N-1})\}^{-1/2s}.
\end{equation}

A combination of the above inequality (\ref{Lemma 6.9.7}) and the
stopping criteria
\begin{equation}\label{Lemma 6.9.9}
\eta_{\mathcal{T}_{N}}^{2}(p_{N},\mathcal{T}_{N})\leq\varepsilon^{2}
\end{equation}
yields the second assertion (\ref{Lemma 6.9.1}).

By raising the second result (\ref{Lemma 6.9.1}) to the $s-$th power
and recording, we obtain
\begin{equation}\label{Lemma 6.9.8}
\varepsilon\lesssim
C^{s}\Theta^{s}(s,\theta)|(p,f,A)|_{s}(\#\mathcal{T}_{N}-\#\mathcal{T}_{0})^{-s}.
\end{equation}
Since $\{\mathcal{E}_{N}^{2}+{\rm
osc}_{\mathcal{T}_{N}}^{2}(p_{N},\mathcal{T}_{N})\}^{1/2}\approx
\eta_{\mathcal{T}_{N}}(p_{N},\mathcal{T}_{N})$, the above inequality
(\ref{Lemma 6.9.8}) and the stopping criteria (\ref{Lemma 6.9.9})
imply the third assertion (\ref{Lemma 6.9.2}).
\end{proof}

Since some results concerning quasi-optimal convergence rate are
obtained under the assumption that $f$ is a piecewise polynomial of
degree $\leq l$ onto $\mathcal{T}_{0}$. By inspiration of these
results, we consider the following algorithm which separates the
oscillation reduction of data $f$ and the quasi error:
\begin{equation*}
\begin{array}{lll}
&{\rm Step}\ 1&\ \displaystyle\ [\mathcal{T}_{H},f_{H}]={\rm
APPROX}(f,\mathcal{T}_{0},\varepsilon/2)\vspace{2mm},\\
&{\rm Step}\ 2\ &\displaystyle\ [\mathcal{T}_{N},(p_{N},u_{N})]={\rm
AMFEM}(\mathcal{T}_{H},f_{H},\varepsilon/2,\theta).
\end{array}
\end{equation*}
The advantage of separating data error and discretization error is
that in second Step 2, oscillation of data $f$ is always zero since
the input data $f_{H}$ is piecewise polynomial of degree $\leq l$
over the initial mesh $\mathcal{T}_{H}$ for AMFEM.

We are now at the position to show the quasi-optimal convergence
rate.

\begin{theorem}\label{thm 6.10}  ({\rm Quasi-optimal convergence rate}) 
For any $f\in L^{2}(\Omega)$, a shape regular triangulation
$\mathcal{T}_{0}$ and a tolerance $\varepsilon>0$. Let
$\theta\in(0,\theta_{*})$, $[\mathcal{T}_{N},(p_{N},u_{N})]={\rm
AMFEM}(\mathcal{T}_{0},f,\varepsilon,\theta)$, and ${\rm
osc}_{N}^{2} :={\rm
osc}_{\mathcal{T}_{N}}^{2}(p_{N},\mathcal{T}_{N})$. If
$(p,f,A)\in\mathbb{A}_{s}$ and $f\in\mathcal{A}_{0}^{s}$, then there
exist a positive constant $C$ depending only on data $(A,f)$,
refinement depth $b$, and $\mathcal{T}_{0}$, but independent of $s$,
such that
\begin{equation}\label{Theorem 6.10}
(\mathcal{E}_{N}^{2}+{\rm osc}_{N}^{2})^{1/2}\leq
C^{s}\Theta^{s}(s,\theta)(|(p,f,A)|_{s}+||f||_{\mathcal{A}_{0}^{s}})(\#\mathcal{T}_{N}-
\#\mathcal{T}_{0})^{-s}.
\end{equation}
\end{theorem}
\begin{proof} By Lemma \ref{lem 6.5}, there exists a
triangulation $\mathcal{T}_{k}$ such that
\begin{equation}\label{Theorem 6.10.1}
{\rm osc}(f,\mathcal{T}_{k})\leq\varepsilon,\ \ {\rm and}\ \
\#\mathcal{T}_{k}-\#\mathcal{T}_{0}\leq
C||f||_{\mathcal{A}_{0}^{s}}^{1/s}\varepsilon^{-1/s}.
\end{equation}
Let $f_{k} :=\Pi_{L_{k}}f$ denote the $L^{2}-$projection of $f$ over
$L_{k}$, and $(\tilde{p},\tilde{u})$ be a pair of solutions of (\ref{quasi-optimality
2.1}) with respect to the right-hand side term $f_{k}$. By Lemma \ref{lem 6.3}
(stable result), we obtain
\begin{equation}\label{Theorem 6.10.2}
||A^{-1/2}(p-\tilde{p})||_{L^{2}(\Omega)}\leq C_{0}^{1/2}{\rm
osc}(f,\mathcal{T}_{k})\leq C_{0}^{1/2}\varepsilon.
\end{equation}

By the definition of $\mathbb{A}_{s}$, and the assumption that
$(p,f,A)\in\mathbb{A}_{s}$, it follows
$(\tilde{p},f_{k},A)\in\mathbb{A}_{s}$ and
\begin{equation}\label{Theorem 6.10.3}
|(\tilde{p},f_{k},A)|_{s}\lesssim|(p,f,A)|_{s}+||f||_{\mathcal{A}_{0}^{s}}.
\end{equation}
We then use the stopping criteria (\ref{Lemma 6.9.9}) and apply
Theorem \ref{thm 6.9} to $(\tilde{p},\tilde{u})$ to obtain
\begin{equation}\label{Theorem 6.10.4}
\begin{array}{lll}
||A^{-1/2}(\tilde{p}-p_{N})||_{L^{2}(\Omega)}^{2}&+&||h_{N}{\rm
div}(\tilde{p}-p_{N})||_{L^{2}(\Omega)}^{2}\vspace{2mm}\\
&+&{\rm osc}_{\mathcal{T}_{N}}^{2}(p_{N},\mathcal{T}_{N})
\lesssim\eta_{\mathcal{T}_{N}}^{2}(p_{N},\mathcal{T}_{N})\leq\varepsilon^{2}
\end{array}
\end{equation}
and
\begin{equation}\label{Theorem 6.10.5}
\#\mathcal{T}_{N}-\#\mathcal{T}_{k}\lesssim
C\Theta(s,\theta)|(\tilde{p},f_{k},A)|_{s}^{1/s}\varepsilon^{-1/s}.
\end{equation}

Notice that $||h_{N}{\rm
div}(p-\tilde{p})||_{L^{2}(\Omega)}\leq||h_{k}{\rm
div}(p-\tilde{p})||_{L^{2}(\Omega)}={\rm osc}(f,\mathcal{T}_{k})$,
and $\Theta(s,\theta)>1$. A combination of (\ref{Theorem
6.10.1})-(\ref{Theorem 6.10.5}) yields
\begin{equation}\label{Theorem 6.10.6}
\begin{array}{lll}
&\ &\mathcal{E}_{N}^{2}+{\rm
osc}_{\mathcal{T}_{N}}^{2}(p_{N},\mathcal{T}_{N})\leq
2||A^{-1/2}(p-\tilde{p})||_{L^{2}(\Omega)}^{2}+
2||h_{N}{\rm div}(p-\tilde{p})||_{L^{2}(\Omega)}^{2}\vspace{2mm}\\
&\ &\ \ +2||A^{-1/2}(\tilde{p}-p_{N})||_{L^{2}(\Omega)}^{2}
+2||h_{N}{\rm div}(\tilde{p}-p_{N})||_{L^{2}(\Omega)}^{2}+
{\rm osc}_{\mathcal{T}_{N}}^{2}(p_{N},\mathcal{T}_{N})\vspace{2mm}\\
&\ &\ \ \lesssim2C_{0}\varepsilon^{2}+2{\rm
osc}^{2}(f,\mathcal{T}_{k})+2\varepsilon^{2}\lesssim\varepsilon^{2}
\end{array}
\end{equation}
and
\begin{equation}\label{Theorem 6.10.7}
\begin{array}{lll}
\#\mathcal{T}_{N}-\#\mathcal{T}_{0}&=&\#\mathcal{T}_{N}-\#\mathcal{T}_{k}+
\#\mathcal{T}_{k}-\#\mathcal{T}_{0}\vspace{2mm}\\
&\lesssim&C\Theta(s,\theta)|(\tilde{p},f_{k},A)|_{s}^{1/s}\varepsilon^{-1/s}+
C||f||_{\mathcal{A}_{0}^{s}}^{1/s}\varepsilon^{-1/s}\vspace{2mm}\\
&\lesssim&C\Theta(s,\theta)(|(p,f,A)|_{s}+
||f||_{\mathcal{A}_{0}^{s}})^{1/s}\varepsilon^{-1/s}.
\end{array}
\end{equation}
The above inequality (\ref{Theorem 6.10.7}) implies
\begin{equation}\label{Theorem 6.10.8}
\varepsilon\lesssim C^{s}\Theta^{s}(s,\theta)(|(p,f,A)|_{s}+
||f||_{\mathcal{A}_{0}^{s}})(\#\mathcal{T}_{N}-\#\mathcal{T}_{0})^{-s}.
\end{equation}
The desired result (\ref{Theorem 6.10}) then follows from the above
two inequalities (\ref{Theorem 6.10.6}) and (\ref{Theorem 6.10.8}).
\end{proof}

\section{Conclusions}
It is  time to recall the main results in this paper. Firstly, we
have removed the restriction of the coefficient matrix of the PDEs
which is required in Carstensen' work, and have analyzed the
efficiency of the a posteriori error estimator obtained by
Carstensen, for RT, BDM, and BDFM elements . Secondly, For the
AMFEM, as
  is customary in practice, the AMFEM marks exclusively according to
  the error estimator and performs a minimal element refinement
  without the interior node property. We have proved that the sum of the
  stress variable error in a weighted norm and the scaled error
  estimator, reduces with a fixed factor between two successive
  adaptive loops, up to an oscillation of data $f$. This geometric
  decay is instrumental to obtain the optimal cardinality of the
  AMFEM. Finally, we have shown that the stress variable error in a weighted norm plus
  the oscillation of data yields a decay rate in
  terms of the number of degrees of freedom as dictated by the best
  approximation for this combined nonlinear quantity, namely we have
  obtain the quasi-optimal convergence rate for the AMFEM.

\end{document}